\documentclass{amsart}
\usepackage{amssymb, amsmath}

\newtheorem{theorem}{Theorem}[section]
\newtheorem{proposition}[theorem]{Proposition}
\newtheorem{lemma}[theorem]{Lemma}
\newtheorem{corollary}[theorem]{Corollary}
\newtheorem{fact}[theorem]{Fact}

\theoremstyle{definition}
\newtheorem{definition}[theorem]{Definition}

\theoremstyle{remark}
\newtheorem{remark}[theorem]{Remark}

\input xy

\xyoption{all}

\numberwithin{equation}{section}

\def \P {\mathbb{P}^1}

\def \C {\mathbb{C}}

\def \ZZ {\mathbb{Z}}

\def \tt {\otimes}
\def \p {\oplus}

\newcommand{\OOO}{\mathcal{O}_{\mathbb{P}^1}}
\newcommand{\OO}{\mathcal{O}}

\newcommand{\MM}{\mathcal{M^{\text{co}}}}
\newcommand{\MPP}{\mathcal{M^{\text{co}}}}

\newcommand{\MPdeg}{\mathcal{M}_{\mathbb{P}^1}^{\text{co}}}
\newcommand{\MMM}{\mathcal{M^{\text{co}}}(2)}

\newcommand{\U}{\mathcal{U}}
\newcommand{\del}{\partial}
\newcommand\numberthis{\addtocounter{equation}{1}\tag{\theequation}}
\DeclareMathOperator{\HH}{H}

\DeclareMathOperator{\End}{End}
\DeclareMathOperator{\SSS}{S}
\DeclareMathOperator{\gr}{gr}
\DeclareMathOperator{\ch}{char}

\DeclareMathOperator{\Aut}{Aut}
\DeclareMathOperator{\Hom}{Hom}

\DeclareMathOperator{\Tot}{Tot}
\DeclareMathOperator{\Pic}{Pic}
\DeclareMathOperator{\I}{Im}
\DeclareMathOperator{\K}{Ker}
\DeclareMathOperator{\Id}{Id}
\DeclareMathOperator{\pr}{pr}
\DeclareMathOperator{\Ext}{Ext^1}
\DeclareMathOperator{\h}{h}
\DeclareMathOperator{\rk}{rk}
\DeclareFontFamily{U}{matha}
{\hyphenchar\font45}
\DeclareFontShape{U}{matha}{m}
{n}{
<5> <6> <7> <8> <9> <10> 
gen * matha 
<10.95> matha10 <12>
<14.4> <17.28> <20.74> <24.88> 
matha12
}{ }

\begin{document}

\title{Moduli spaces of semistable rank-2 co-Higgs bundles over $\P \times \P$}
\author{Alejandra Vicente Colmenares}

\address{Department of Pure Mathematics, University of Waterloo. 200 University Ave W, Waterloo, On. Canada,  N2L 3G1}
\email{avicente@uwaterloo.ca}
\date{\today}

\begin{abstract} 
It has been observed, by S. Rayan, that the complex projective
surfaces that potentially admit non-trivial examples of semistable
co-Higgs bundles must be found at the lower end of the Enriques-Kodaira
classification. Motivated by this remark, it is natural to study the
geometry of these objects over $\P \times \P$. In this paper, we present necessary and sufficient conditions on the Chern
classes $c_1,c_2$ of a bundle that guarantee the non-emptiness of the
moduli space $\MPP(c_1,c_2)$ of rank 2 semistable co-Higgs bundles over
$\P \times \P$ (we do this with respect to the
standard polarization). We also
give an explicit description of the moduli spaces for certain choices of
$c_1$ and $c_2$.\end{abstract}

\maketitle

\section{Introduction}

A Higgs bundle on a complex projective manifold $X$ is a pair $(E, \Phi)$ consisting of a holomorphic vector bundle $E$ over $X$ together with a Higgs field $\Phi\colon E \to E \tt T^{\lor}$ taking values in the holomorphic cotangent bundle $T^{\lor}$ of $X$ such that $\Phi \wedge \Phi \in \HH^0(\End E \tt \wedge^2 T^{\lor})$ is identically zero. Higgs bundles were introduced almost 30 years ago by Hitchin in \cite{NH4} and by Simpson in his PhD dissertation \cite{Simthesis}. These objects have several interesting applications to both physics and mathematics, and have been extensively studied by many other authors, including Bradlow, Garc{\'i}a-Prada, Gothen, Wentworth; see for instance \cite{BGG1, BGG, GP, G, G1, NH2, Sim, Simp, Went}. Co-Higgs bundles, on the other hand, are holomorphic vector bundles $E$ paired with Higgs fields $\Phi \colon E \to E \tt T$ taking values in the holomorphic tangent bundle $T$ of $X$, instead of its holomorphic cotangent bundle $T^{\lor}$, and satisfying the same integrability condition $\Phi \wedge \Phi=0 \in \HH^0(\End E \tt \wedge^2 T)$. 

The study of co-Higgs bundles is fairly recent. They first appeared in the work of Gualtieri \cite{MG}, and were further studied by Hitchin in \cite{NH5, NH3} and Rayan in \cite{SR1, SR2, SR3}. As Rayan pertinently points out in his PhD dissertation \cite{SR1}, the study of co-Higgs bundles goes beyond idle curiosity, as these objects appear naturally in geometry; for example, in generalized complex geometry and in the theory of twisted quiver bundles. In the realm of generalized complex geometry, as introduced by Hitchin in \cite{NH1} and developed by Gualtieri in \cite{MG}, co-Higgs bundles emerge as generalized holomorphic vector bundles over complex manifolds (regarded as generalized complex manifolds). Co-Higgs bundles, just as Higgs bundles, also fit in the realm of twisted quiver bundles as developed by \'{A}lvarez-C\'{o}nsul and Garc\'{i}a-Prada in \cite{OGP1}. Indeed, a co-Higgs bundle can be thought of as a quiver bundle formed by one vertex and one arrow (with the homomorphism satisfying the integrability condition) whose head and tail coincide, and the twisting bundle is the holomorphic cotangent bundle $T^{\lor}$.

Co-Higgs bundles come with a natural stability condition (see Definition \ref{definitionstabilityofcH}), analogous to the one discovered by Hitchin in \cite{NH4} for Higgs bundles, which allows the study of their moduli spaces. Rayan has already given a complete characterization of (rank 2) semistable co-Higgs bundles over Riemann surfaces, but very little is known about these objects in higher dimensions. In his PhD dissertation \cite{SR1} and in \cite{SR2, SR3}, Rayan makes a thorough investigation of semistable rank 2 co-Higgs bundles over the Riemann sphere and constructs some examples over the projective plane. He also proves a non-existence result for non-trivial (i.e., non-zero Higgs field) stable co-Higgs bundles over K3 and general type surfaces, suggesting that some of the interesting examples must be found at the lower end of the Enriques--Kodaira classification of (compact) complex surfaces. 

Motivated by the above observation, in this paper, we focus on $\P \times \P$. After fixing the standard polarization for the notion of stability, we first give necessary and sufficient conditions on the second Chern class of the bundle that guarantee the existence of non-trivial semistable rank 2 co-Higgs pairs over $\P \times \P$. We then proceed to explicitly describe some of the moduli spaces. 

For the benefit of the reader we now give a brief account of the main results of the paper. In section \ref{section4}, we prove an existence theorem (Theorem \ref{fullanswernon-emptiness}) on semistable rank 2 co-Higgs bundles over $\P \times \P$. That is, given first and second Chern classes, $c_1$ and $c_2$, we give conditions that ensure the existence of a non-trivial semistable co-Higgs pair with those Chern classes. Indeed, if we let $F$ and $C_0$ denote the two classes of divisors that freely generate $\Pic(\P \times \P)$, and we let $\MM(c_1,c_2)$ denote the moduli space of rank 2 semistable co-Higgs bundles over $\P \times \P$ with first Chern class $c_1$ and second Chern class $c_2$, we prove: 

\vspace{.05in}

\noindent {\bf Theorem.} {\it Let $c_1=\alpha C_0 + \beta F$ and $c_2=\gamma$. Then, the moduli space $\MM(c_1,c_2)$ is non-empty (and moreover it contains a non-trivial co-Higgs pair) if and only if one of the following holds:}
\begin{enumerate}
{\it 
\item at least one of $\alpha$ and $\beta$ is even and $2\gamma \geq \alpha\beta$;
\item $\alpha$ and $\beta$ are both odd and $2\gamma \geq \alpha \beta -2$.}
\end{enumerate}

\vspace{.05in}

In Section \ref{section5}, we give explicit descriptions of the moduli spaces $\MPP(c_1,c_2)$ for certain values of $c_1$ and $c_2$. For $c_1=-F$ (or $-C_0$) and $c_2=0$, we have the following description of the moduli space (Theorem \ref{explicitdescriptionmodulispace0}):

\vspace{.05in}

\noindent {\bf Theorem.} {\it
The moduli space $\MPP(-F,0)$ is a 6-dimensional smooth variety isomorphic to the moduli space $\MPdeg(-1)$ of rank 2 stable co-Higgs bundles of degree $-1$ over $\P$ (the latter is described in \cite[Section~7]{SR2})}. 

\vspace{.05in}

For $c_1=0$ and $c_2=0$, we are not able to give such an explicit description. Nonetheless, we show that there are only three underlying bundles that admit semistable Higgs fields: 
$$\OO \p \OO, \quad \OO(1,0) \p \OO(-1,0) \quad \text{and}\quad \OO(0,1) \p \OO(0,-1).$$ 
Also, we fully describe the Higgs fields that make $\OO \p \OO$ strictly semistable, and in Proposition \ref{modulispace00} we prove that the Higgs fields of points in $\MPP(0,0)$ with underlying bundle $\OO(1,0) \p \OO(-1,0)$  ($\OO(0,1) \p \OO(0,-1)$) are naturally parametrized by $\HH^0(\OO(4,0))$ ($\HH^0(\OO(0,4))$, respectively). 

For $c_1=-F$ and $c_2=1$, we show that any underlying bundle in the moduli space $\MM(-F,1)$ is an extension of $\OO(-1,1)$ by $\OO(0,-1)$ (Proposition \ref{allunderlying1}). Then, for the sake of being explicit, we describe all the Higgs fields that these bundles admit. Finally, we give a full description of the moduli space in Theorem \ref{explicitdescriptionmodulispace1} (see also Propositon \ref{modulispace1beforequotiening}):

\vspace{.05in}

\noindent {\bf Theorem.} {\it
The moduli space $\MM(-F,1)$ is a 7-dimensional algebraic variety whose singular locus are the points $(E,0)$ for any non-trivial extension $E$ of $\OO(-1,1)$ by $\OO(0,-1)$.}

\vspace{.05in}

An important tool in constructing examples of semistable co-Higgs pairs or in understanding their moduli spaces are spectral covers, and so we conclude this paper with a short section that addresses this construction and the Hitchin correspondence in the case of $\P \times \P$ (Section \ref{ssP1xP1}).  Given a rank 2 co-Higgs bundle over the complex projective manifold $X$, we can associate to it a spectral manifold, which is a double cover of $X$ naturally living in the total space of its tangent bundle.  By the work of Hitchin and Simpson, it is well known that, under certain genericity conditions, one can construct rank 2 stable co-Higgs pairs over $X$ in the following fashion: Take any rank 1 torsion-free coherent sheaf $\mathcal{F}$ over the spectral cover of $X$ and push it down to obtain the underlying bundle of the co-Higgs pair, take also the push down of the multiplication map associated to $\mathcal{F}$ to obtain the Higgs field (one would of course need to check that the integrability condition is satisfied). In the $\P \times \P$ setting, we show that the generic elements in the moduli space are such that the underlying bundles are not decomposable (Proposition \ref{fibresoftheHitchinmap}). Finally, in certain non-generic cases, we are able to describe the spectral covers as trivial elliptic fibrations over $\P$; in these cases, the fibres of the Hitchin map do contain co-Higgs pairs with decomposable underlying bundles (Proposition \ref{correspondencepushingdownbundles}). 

\

\noindent {\bf Acknowledgments.} I thank Ruxandra Moraru for introducing me to co-Higgs bundles, and for several discussions on the topic. I also thank Steven Rayan for the many helpful and interesting conversations on this subject. 

\section{Preliminaries}\label{co-Higgsbundles}

In this section we recall the basic definitions and properties of (semistable) co-Higgs bundles 
over complex projective manifolds. For more details we refer the reader to \cite{SR1}. Throughout this section we fix a complex projective manifold $X$ and denote its tangent bundle by $T$.

\begin{definition}
A \emph{co-Higgs bundle} (or \emph{co-Higgs pair}) on $X$ is a vector bundle $V~\to~X$ together with a map $\Phi \in \HH^0(X;\End V \tt T)$ for which $\Phi \wedge \Phi \in \HH^0(X;\End V \otimes \wedge^2 T)$ is identically zero. We refer to such a $\Phi$ as a  \emph{Higgs field} of $V$. 
\end{definition}


\begin{remark}\label{superwedge}
Let us recall that the wedge, $- \wedge -$, acts as the commutator in elements of $\End V$ and as the usual wedge in elements of $T$. For instance, when $X$ is a surface and $\Psi, \Phi \in \HH^0(\End V \tt T)$, if we work locally, say $\Psi= \Psi_1\del_1 + \Psi_2\del_2$ and $\Phi=\Phi_1\del_1+\Phi_2\del_2$, then
\begin{align*}
\Psi \wedge \Phi 
                           & =  ([\Psi_1,\Phi_2]-[\Psi_2,\Phi_1])\del_1\wedge\del_2.
\end{align*}
\end{remark}

In order to define a natural stability condition for co-Higgs bundles, let us fix a polarization $H$ in the ample cone of $X$.

\begin{definition}\label{definitionstabilityofcH}
A co-Higgs bundle $(V, \Phi)$ on a complex projective manifold $X$ is stable (respectively, semistable) if
\begin{equation}\label{sta}
\mu(W) := \frac{\deg_H W}{\rk W} < \frac{\deg_H V}{\rk V} =: \mu(V)
\end{equation}
(respectively, $\leq$) for each non-zero proper subsheaf $W$ of $V$ that is $\Phi$-invariant; i.e., $\Phi(W) \subseteq W \otimes T$. The number $\mu(V)$ is called the $H$-slope of $V$.
\end{definition}

\begin{remark} \
\begin{enumerate}
\item Recall that, if $X$ is a surface and $V$ is of rank 2, it suffices to check inequality~(\ref{sta}) for sub-line bundles of $V$.
\item We will repeatedly use the fact that tensoring a (semi)stable co-Higgs bundle by a line bundle does not affect (semi)stability.
\end{enumerate}
\end{remark}


Let $(V,\Phi)$ be a co-Higgs bundle. We say that the co-Higgs pair is  \emph{non-trivial} if $\Phi$ is non-zero. Also, we say that $\Phi$ is (semi)stable whenever the pair $(V,\Phi)$ is (semi)stable. Recall that the trace-free part of  $\Phi$ is 
\begin{displaymath}
\Phi_0:= \Phi-\left(\frac{\text{Tr}\Phi}{\rk(V)}\right)\Id \in \HH^0(\End_0 V \tt T).
\end{displaymath}
It is immediate to check that $(V,\Phi)$ is (semi)stable if and only if $(V,\Phi_0)$ is (semi)stable. Hence, from now on, whenever we have a co-Higgs pair $(V,\Phi)$, $\Phi$ is assumed to be trace-free. 

A morphism of co-Higgs bundles $(V,\Phi)$ and $(V',\Phi')$ is a commutative diagram

\begin{displaymath}
\xymatrix{
V \ar[d]^{\Phi}\ar[rr]^{\psi}&&V' \ar[d]_{\Phi'} \\
V \otimes T \ar[rr]^{\psi \otimes \Id}&&V' \otimes T 
}
\end{displaymath}

\noindent in which $\psi: V \to V'$ is a homomorphism of vector bundles. The pairs $(V,\Phi)$ and $(V',\Phi')$ are said to be isomorphic if $\psi$ is an isomorphism of vector bundles. 

In the moduli spaces of semistable co-Higgs bundles over $X$, we identify co-Higgs pairs up to $S$-equivalence. That is, pairs whose associated graded objects are isomorphic as co-Higgs bundles. Let us briefly recall this notion. Given a  strictly semistable pair $(V,\Phi)$, we obtain a co-Higgs Jordan-H\"older filtration:
\begin{displaymath}
0=V_0 \subset \cdots \subset V_m=V
\end{displaymath}
for some $m$. Here $(V_j, \Phi)$ is semistable for $1 \leq i \leq m-1$, $(V_j/V_{j-1}, \Phi)$ is stable, and $\mu(V_j)=\mu(V_j/V_{j-1})=\mu(V)$ for $1 \leq j \leq m$. The \emph{associated graded object} of $(V, \Phi)$ is:
\begin{displaymath}
\gr(V, \Phi) := \bigoplus_{j=1}^{m}(V_j/V_{j-1}, \Phi).
\end{displaymath}

The associated graded object of a strictly semistable pair $(V, \Phi)$ is easy to describe when the underlying bundle is decomposable:

\begin{lemma}\label{gr}
Let $X$ be a surface, and let $E=G_1 \p G_2$ be a decomposable rank 2 bundle over $X$. Suppose $(E,\Phi)$ is strictly semistable with $\Phi=\left(\begin{array}{cc}
A & B \\
C & -A \\
\end{array}\right) \in \HH^0(\End_0 E \tt T)$. If $G_1$ is $\Phi$-invariant and $\mu(G_1)=\mu(E)$, then 
\begin{displaymath}
\gr\left(E, \Phi \right)=\left(E, \left(\begin{array}{cc}
A & 0 \\
0 & -A \\
\end{array}\right) \right).
\end{displaymath}
\end{lemma}

\begin{proof}
The following is a co-Higgs Jordan-H\"older filtration of $E$:
\begin{displaymath}
0 \subset G_1  \subset E.
\end{displaymath}
Hence
\begin{displaymath}
\gr(E,\Phi) = (G_1, \Phi_1) \p (G_2, \Phi_2),
\end{displaymath}
with $\Phi_1=A$ and $\Phi_2=-A$.
\end{proof}

Let us finish this section with the following two lemmas, which will be used in subsequent sections.

\begin{lemma}\label{isb}
Let $W$ be a sub-bundle of $V$ and
\begin{displaymath}
S:=\{\varphi \in \HH^0(X;{\Hom}(W, V \otimes T)) \;\; | \;\;  \varphi=\Phi |_{W} \text{ for some } \Phi \in \HH^0(X;\End V \otimes T) \}.
\end{displaymath}
Also, let $\boldsymbol \iota: \HH^0(X;\Hom(W,W\tt T)) \to \HH^0(X;\Hom(W,V \tt T))$ be the map induced by the inclusion $\iota: W \hookrightarrow V$ . If $S \subseteq \I(\boldsymbol \iota$), then $W$ is $\Phi$-invariant for any $\Phi \in \HH^0(X;\End V \otimes T)$.
\end{lemma}

\begin{proof}
Take $\Phi \in \HH^0(X;\End V \otimes T)$ and consider $\Phi |_{W}$, which is an element of $S$. Since $S \subseteq \I(\boldsymbol \iota)$, we must have that $\Phi|_{W}=\boldsymbol \iota(\psi)$ for some $\psi \in \HH^0(X;\Hom(W,W \otimes T))$. 
Hence $W$ is $\Phi$-invariant.
\end{proof}


\begin{lemma}\label{indeedsemistable}
Let $E=G_1 \p G_2$ be a decomposable rank 2 bundle over $X$. If $\mu(G_1) > \mu(E)$ and $\HH^0(G_1^{\lor} \tt G_2 \tt T) \neq 0$, then there exists a Higgs field $\Phi$ such that $(E,\Phi)$ is semistable. Moreover, any Higgs field with non-zero $(2,1)$-entry makes $(E, \Phi)$ into a semistable pair.
\end{lemma}

\begin{proof}
Note that having $\mu(G_1) > \mu(E)$ implies that $E$ is unstable, with $G_1$ being the unique sub-line bundle that destabilizes $E$ (see \cite[Chapter~4]{Fr}). Any Higgs field is an element of $\HH^0(\End_0 E \tt T)$ which is integrable. In particular,
\begin{displaymath}
\Phi=\left(\begin{array}{cc}
A & B \\
C & -A \\
\end{array}\right),
\end{displaymath}
with $A \in \HH^0(T)$, $B \in \HH^0(G_1 \tt G_2^{\lor} \tt T)$ and $C \in \HH^0(G_1^{\lor} \tt G_2 \tt T)$. Since $\HH^0(G_1^{\lor} \tt G_2 \tt T) \neq 0$, there exists $\Phi$ with non-zero $C$. In that case, we have that $G_1$ is not $\Phi$-invariant, and so $(E,\Phi)$ is semistable. 
\end{proof}

\section{On bundles over $\P \times \P$.}

In this section we recall some of the properties of (rank 1 and 2) vector bundles over $\P \times \P$ that will be relevant for us. Some good references on the subject are \cite{AB1, AB2, ABM, AM, Br}, which treat the topic in the more general setting of Hirzebruch surfaces. 
From now on, $C_0$ will denote a section of $\pr_1: \P \times \P \to \P$ such that $C_0^2 = 0$, and $F$ will denote a general fibre of $\pr_1$. These two divisors freely generate $\Pic(\P \times \P)$. Recall that a divisor $aC_0 + bF$ on $\P \times \P$ is ample (equivalently, very ample) if and only if $a > 0$ and $b > 0$ (see \cite[Chapter~5]{H}). 

Let $\pr_i$ denote the projection from $\P \times \P$ onto the $i$-th copy of $\P$, then we let
\begin{displaymath}
\OO(a,b):=\pr_1^*\OOO(a) \tt \pr_2^*\OOO(b),
\end{displaymath}
denote the line bundle over $\P \times \P$ corresponding to the divisor $bC_0+aF$. The tangent bundle $T$ of $\P \times \P$ is the rank 2 decomposable bundle:
\begin{displaymath}
T=\OO(2,0) \p \OO(0,2).
\end{displaymath}

\begin{remark}\label{vanishingcohomology}
We remind the reader the conditions on $a$ and $b$ for the cohomology groups of a line bundle over $\P \times \P$ to vanish:

\begin{displaymath}
\begin{array}{ccc}

\vspace{.05in}

\HH^0(\P \times \P, \OO(a,b)) & = 0 & \text{ if and only if } a<0 \text{ or } b<0. \\ 
\vspace{.05in}

\HH^1(\P \times \P, \OO(a,b)) & = 0 & \text{ if and only if } a<0 \text{ and } b<0, \text{ or } a \geq -1 \text{ and } b \geq -1. \\

\HH^2(\P \times \P, \OO(a,b)) & = 0 & \text{ if and only if } a \geq -1 \text{ or } b \geq -1.
\end{array}
\end{displaymath} 
\end{remark}

From now on, unless otherwise specified, the notation $\HH^i(\mathcal{F})$, where $\mathcal{F}$ is a coherent torsion-free sheaf over $\P \times \P$, will stand for $\HH^i(\P \times \P; \mathcal{F})$.

A rank 2 bundle $E$ over $\P \times \P$ is always an extension of the form
\begin{displaymath}
0 \to L_1 \to E \to L_2 \tt I_Z \to 0,
\end{displaymath}
where $Z$ is a finite set of points in $\P \times \P$. In this case, the Chern classes of $E$ are given by
\begin{displaymath}
\begin{array}{c}
c_1(E)=c_1(L_1) + c_1(L_2), \\
c_2(E)=c_1(L_1) \cdot c_1(L_2) + |Z|.
\end{array}
\end{displaymath}
Besides the Chern classes, which determine the topological type of $E$, there are two numerical invariants describing it as an extension in a canonical manner. The first invariant $d_E$ is defined by the splitting type on the general fibre $F$: if $E_{|F} \cong \OOO(d) \p \OOO(d')$ with $d \geq d'$, then $d_E=d$. The second invariant $r_E$ is obtained from a push-forward as follows. Note that the bundle $\pi_*(E(0,-d))$ is either of rank one or two, according to whether $d>d'$ or $d=d'$, respectively. If $d>d'$, we put $r_E=r=\deg((\pr_1)_*(E(0,-d)))$. If $d=d'$, then $(\pr_1)_*(E(0,-d))=\OOO(r) \p \OOO(s)$ with $r \geq s$ and we put $r_E=r$.

Thus, a rank 2 vector bundle $E$ with numerical invariants $d$ and $r$ can be expressed as an extension
\begin{equation}\label{ABM}
0 \rightarrow \OO(r,d) \rightarrow E \rightarrow \OO(r',d') \otimes I_Z \rightarrow 0, 
\end{equation}
where $Z$ is a finite set of points in $\P \times \P$. This extension is unique if either $d > d'$ or $d=d'$ and $s < r$, where $s$ is the extra invariant described above. Moreover, if $c_1=\alpha C_0 + \beta F$, then
\begin{displaymath}
|Z|=\ell(c_1,c_2,d,r):=c_2- \alpha r -\beta d + 2dr.
\end{displaymath}


These numerical invariants help us to better understand rank 2 bundles over $\P \times \P$. Indeed,  let
\begin{displaymath}
M(c_1, c_2, d, r)=\{E \to \P \times \P: c_1(E)=c_1, c_2(E)=c_2, d_E=d, r_E=r \}/\sim,
\end{displaymath}
where $\sim$ denotes the equivalence relation of vector bundle isomorphism, be the coarse moduli space of rank 2 bundles with fixed Chern classes $c_1$ and $c_2$, and fixed numerical invariants $d$ and $r$. The following theorem tells us when this moduli space is non-empty (see \cite[Theorem~2.1]{ABM}, or for more details \cite{AB2}). 

\begin{theorem}\label{ABM1}
Put $c_1=\alpha C_0+\beta F$ . The set $M(c_1, c_2, d, r)$ is non-empty if and only if $\ell:=\ell(c_1,c_2,d,r) \geq 0$ and one of the following conditions is satisfied:
\begin{enumerate}
\item $2d > \alpha$, or
\item $2d = \alpha$, $\beta-2r \leq \ell$.
\end{enumerate}
\end{theorem}




Let us finish this section with a remark on the integrability condition of the Higgs fields in our setting.

\begin{remark}
If $E$ is a rank 2 bundle over $\P \times \P$, then any $\Phi \in \HH^0(\End_0 E \tt T)$ is of the form $\Phi=\Phi_1 + \Phi_2$ with $\Phi_1 \in \HH^0(\End_0 E(2,0))$ and $\Phi_2 \in \HH^0(\End_0 E(0,2))$. Working locally on an open set, where $\End_0 E$ and $T$ are trivial, we can write

\begin{displaymath}
\Phi=\Phi_1+\Phi_2=\left(\begin{array}{cc}
A_1 & B_1 \\
C_1 & -A_1 \\
\end{array}\right)
+
\left(\begin{array}{cc}
A_2 & B_2 \\
C_2 & -A_2 \\
\end{array}\right),
\end{displaymath}
where $A_i, B_i, C_i$ are complex valued functions for $i=1, 2$. Also, on this open set, 
\begin{displaymath}
\Phi \wedge \Phi=2[\Phi_1,\Phi_2],
\end{displaymath} 
so we can locally write
\begin{displaymath}
\Phi \wedge \Phi =2\left(\begin{array}{cc}
B_1C_2-C_1B_2 & 2(A_1B_2-B_1A_2) \\
2(C_1A_2-A_1C_2) & -(B_1C_2-C_1B_2) \\
\end{array}\right).
\end{displaymath}
Thus, we see that $\Phi$ is an (integrable) Higgs field if and only if, in each local trivialization, we have that 
\begin{align*}
B_1C_2 & = C_1B_2  \\  \numberthis \label{integrability}
A_1B_2 & = B_1A_2   \\ 
C_1A_2 & = A_1C_2.
\end{align*}
\end{remark}


\section{On the existence of non-trivial co-Higgs bundles over $\P \times \P$}\label{section4}

In this section, we give a complete answer to the question: Fixing the polarization $H=C_0+F$ in the ample cone of $\P \times \P$, for which values of $c_1$ and $c_2$ are the moduli spaces of semistable rank 2 co-Higgs bundles over $\P \times \P$ non-empty? (See Theorem \ref{fullanswernon-emptiness}). 

From now on, we will fix the standard polarization $H=C_0+F$, which naturally generalizes the notion of degree from $\P$ to $\P \times \P$. Although this may seem as a constraint, it does not impose a severe restriction for lower values of $c_2$, which is the main focus of Section~\ref{section5} where we give explicit descriptions of some moduli spaces. Indeed, we will see that a necessary condition for the existence of semistable co-Higgs bundles over $\P \times \P$ is that $c_2 \geq 0$. Now, when $c_2=0$, one can show that there is only one chamber in the ample cone, so there really is no loss of generality by choosing the standard polarization in this case. When $c_2=1$, there are at most two chambers, and one can use similar arguments to the one presented here to, after fixing an appropriate ample divisor, construct co-Higgs bundles which are semistable with respect to this chosen polarization. For more details on this, we refer the reader to \cite[Chapter~3]{AVC}.

The following lemma allows us to restrict ourselves to certain first Chern classes. 

\begin{lemma}\label{decomposable}
Let $E$ be a rank 2 vector bundle over $\P \times \P$. Then there is a line bundle $L$ such that $c_1(E \tt L)=0$ or $c_1(E \tt L)=-F$ or $c_1(E \tt L)=-C_0$ or $c_1(E)=-C_0-F$.
\end{lemma}

\begin{proof}
Let $c_1(E)=\alpha C_0 + \beta F$. The proof is immediate after choosing $L$ to be one of 
$$\!\OO\! \left(\!- \frac{\! \beta}{2} \!,\! -\frac{\! \alpha}{2}\! \right)\!, \OO \! \left(\! -\! \left(\! \frac{1+\beta}{2}\! \right) \!,\! -\frac{\alpha}{2}\! \right)\!, \OO\! \left(\! -\frac{\beta}{2}\!,\!- \! \left(\! \frac{1+\alpha}{2}\right)\! \right)\!, \OO\! \left(\!-\! \left(\! \frac{1+\beta}{2}\!\right)\!,\!-\!\left(\! \frac{1+\alpha}{2}\! \right)\! \right)\!,$$
depending on the parity of $\alpha$ and $\beta$.

\end{proof}

When we work with a rank 2 vector bundle $E$ over $\P \times \P$, and we tensor it by a line bundle to obtain one of the first Chern classes $0, -C_0, -F$ or $-C_0-F$, which from now on will be referred to as \emph{reduced classes}, we also modify its second Chern class. That is,
\begin{displaymath}
c_2(E \tt L)= c_2(E) + c_1(E) \cdot c_1(L) + c_1(L)^2,
\end{displaymath}
and so we have:

\begin{corollary}\label{modifiedc_2}
Let $E$ be a rank 2 vector bundle over $\P \times \P$ with $c_1(E)=\alpha C_0 + \beta F$ and $c_2(E)=\gamma$. Then, for any line bundle $L$ we have
\begin{enumerate}
\item If $c_1(E \tt L)=0, -F$ or $-C_0$, then $c_2(E \tt L)= \gamma - \frac{\alpha\beta}{2}$.
\item If $c_1(E \tt L)=-C_0-F$, then $c_2(E \tt L)= \gamma + \frac{1-\alpha\beta}{2}$.\end{enumerate}
\end{corollary}

Let us now work with the reduced classes, and give necessary conditions on $c_2$ in order to have a semistable co-Higgs pair.

\begin{theorem}\label{extid}
Let $E$ be a rank 2 vector bundle over $\P \times \P$. Suppose $(E,\Phi)$ is semistable. 
\begin{enumerate}
\item If $c_1(E)=0, -F$ or $-C_0$, then $c_2(E) \geq 0$. 
\item If $c_1(E)=-C_0-F$, then $c_2(E) \geq 1$. 
\end{enumerate}
Furthermore, in every case, when equality holds, $E$ is an extension of line bundles.
\end{theorem}

\begin{proof}
Let $c_1(E)=\alpha C_0 + \beta F$, where $(\alpha,\beta) \in \left\{(0,0),(0,-1),(-1,0), (-1,-1) \right\} $. Hence, by (\ref{ABM}), $E$ fits into an exact sequence of the form
\begin{equation}\label{ses01}
0 \rightarrow \OO(r,d) \rightarrow E \rightarrow \OO(\beta-r,\alpha-d) \otimes I_Z \rightarrow 0,
\end{equation}
where $Z$ is a finite set of points in $\P \times \P$. Then, $c_2(E)=d(\beta -2r) + r\alpha + \ell(Z)$.  

Let us now work by cases:
\begin{enumerate}
\item [(i)] $(\alpha,\beta)=(0,0)$: Since $E_{|F} \cong \OOO(d) \p \OOO(-d)$, we have that $d \geq 0$. In this case, $c_2(E)=-2dr+\ell(Z)$. Towards a contradiction, assume that $c_2(E) < 0$, or $c_2(E)=0$ and $\ell(Z) > 0$. We then have that $d>0$ and $r > 0$.
\item [(ii)] $(\alpha,\beta)=(0,-1)$: As in case (i), $d \geq 0$, but now $c_2(E)=d(-1-2r)+\ell(Z)$. Towards a contradiction, assume that $c_2(E) < 0$, or $c_2(E)=0$ and $\ell(Z) > 0$. We then have that $d>0$ and $r \geq 0$.
\item [(iii)] $(\alpha,\beta)=(-1,0)$: Since $E_{|F} \cong \OOO(d) \p \OOO(-1-d)$, we have that $d\geq -1-d$, and so $d \geq 0$. In this case, $c_2(E)=-r(2d+1)+\ell(Z)$. Towards a contradiction, assume that either $c_2(E) < 0$, or $c_2(E)=0$ and $\ell(Z) > 0$. We then have that $r > 0$.
\item [(iv)] $(\alpha,\beta)=(-1,-1)$: As in case (iii), $d \geq 0$, but now and $c_2(E)=d(-1-2r)-r+\ell(Z)$. Towards a contradiction, assume that either $c_2(E) < 1$, or $c_2(E)=1$ and $\ell(Z) > 0$. We then have that $r \geq 0$.
\end{enumerate}
Now, since $T=\OO(2,0) \p \OO(0,2)$, by plugging in the corresponding values of $(\alpha,\beta)$, and the corresponding bounds on $d$ and $r$ described in (i) to (iv) above, one can easily check that, in all four cases,
\begin{align*}
\HH^0(\OO(\beta -2r,\alpha -2d) \otimes T \otimes I_Z) & =  \HH^0(\OO(\beta -2r+2,\alpha -2d) \otimes I_Z) \\
& \p  \HH^0(\OO(\beta -2r,\alpha -2d+2) \otimes I_Z) \\
& = 0.
\end{align*}

By tensoring (\ref{ses01}) with $\OO(r,d)^{\lor} \otimes T$ and passing to the long exact sequence in cohomology, we get
\begin{displaymath}
0 \rightarrow \HH^0(T) \rightarrow \HH^0(\OO(r,d)^{\lor} \otimes E \otimes T) \rightarrow \HH^0(\OO(\beta -2r,\alpha -2d) \otimes T \otimes I_Z) \rightarrow \dots.
\end{displaymath}

However, since $\HH^0(\OO(\beta -2r,\alpha -2d) \otimes T \otimes I_Z)=0$ we get that $\HH^0(T) \cong \HH^0(\OO(r,d)^{\lor} \otimes E \otimes T)$, which, by Lemma \ref{isb}, implies that $\OO(r,d)$ is $\Phi$-invariant for any $\Phi \in$ H$^0(\End_0 E \otimes T)$. Furthermore, note that $\mu(\OO(r,d)) = r + d$, while $$\mu(E)=\frac{\alpha + \beta}{2}.$$
Thus, we have that
\begin{align*}
\mu(\OO(r,d)) - \mu(E) & =  r+d-\left(\frac{\alpha + \beta }{2} \right),
\end{align*}
which in all four cases, is a strictly positive number. This contradicts semistability of $(E,\Phi)$, and the result follows. 
\end{proof}

From now on, we will denote the moduli space of rank 2 semistable co-Higgs bundles with fixed Chern classes $c_1$ and $c_2$ by $\MM(c_1,c_2)$. We will soon see (Theorem \ref{nonemptinessmayorados}) that the necessary conditions imposed on $c_2$ to guarantee the existence of semistable co-Higgs pairs in $\MM(c_1, c_2)$, presented in Theorem \ref{extid} above, are indeed sufficient.

In Theorem \ref{nonemptinessmayorados}, in most cases, we show that whenever the moduli space $\MM(c_1,c_2)$ is non-empty, it actually contains a semistable bundle $E$, which we can clearly equip with the zero Higgs field in order to yield a semistable co-Higgs pair. However, we want to exhibit non-trivial co-Higgs bundles, showing that these objects do constitute an enlargement of the class of semistable bundles. In order to so, we show that $\HH^0(\End_0 E \tt T)= \HH^0(\End_0 E (2,0)) \p \HH^0(\End_0 E (0,2)) \neq 0$. Indeed, in all the examples we construct on the proof of this theorem, we will see that $\HH^0(\End_0 E (2,0)) \neq 0$. To prove this, we will need the proposition below. We first recall a basic fact about modules.

\begin{fact}\label{basicfactformodulos}
Let
\begin{displaymath}
0 \to A \xrightarrow{\iota} B_1 \p B_2 \xrightarrow{p} C \to 0
\end{displaymath}
be an exact sequence of $R$-modules. If $A \neq 0$ and $p_{|B_2}$ is injective, then $B_1 \neq 0$.
\end{fact}

\begin{proposition}\label{toshowHiggsfields}
Let $E$ be a rank 2 vector bundle over $\P \times \P$ that fits into the exact sequence
\begin{equation}\label{extension}
0 \to L_1 \xrightarrow{\iota} E \xrightarrow{p} L_2 \tt I_Z \to 0.
\end{equation}
If $\HH^0(L_2^{\lor} \tt E(2,0)) \neq 0$, then $\HH^0(\End_0 E (2,0)) \neq 0$.
\end{proposition}

\begin{proof}
Start by taking the dual of the exact sequence (\ref{extension}), tensor it by $E(2,0)$ and pass to the exact sequence in cohomology to get
\begin{equation*}
0 \to \HH^0(L_2^{\lor} \tt E(2,0)) \to \HH^0(\End E(2,0)) \to \HH^0(L_1^{\lor} \tt E(2,0) \tt I_Z) \to \dots
\end{equation*}
The map $\HH^0(\End E \tt \OO(2,0)) \to \HH^0(L_1^{\lor} \tt E \tt I_Z \tt \OO(2,0))$ is the induced one from 
\begin{displaymath}
\epsilon \tt \Id_{\OO(2,0)}: \End E \tt \OO(2,0) \to (L_1^{\lor} \tt E \tt I_Z) \tt \OO(2,0),
\end{displaymath}
where $\epsilon$ takes $h$ to $h \circ \iota$. Writing $\End E = \End_0 E \p \OO$, and noting that $\Id_E$ generates $\OO$ in $\End E$, we get that the map induced by $\epsilon \tt \Id_{\OO(2,0)}$, 
restricted to $\HH^0(\OO \tt \OO(2,0))$, is injective. By Lemma \ref{basicfactformodulos}, $\HH^0(\End_0 E (2,0)) \neq 0$, as desired.
\end{proof}

The next lemma shows that some of the bundles that we construct in the proof of Theorem \ref{nonemptinessmayorados} are indeed stable.

\begin{lemma}\label{stabilityfor-Fand-C_0}
Let $E$ be such that it either fits into the exact sequence
\begin{displaymath}
0 \to \OO(-1,0) \to E \to I_Z \to 0,
\end{displaymath} 
with $\ell(Z) \geq 1$,
or 
\begin{displaymath}
0 \to \OO(-1,0) \to E \to \OO(1,-1) \tt I_Z \to 0,
\end{displaymath} 
with $\ell(Z) \geq 1$.
Then $E$ is stable.
\end{lemma}

\begin{proof}
In the first case, we prove that $E$ is stable by showing that, if there is a non-zero map $\OO(a,b) \to E$, then
\begin{displaymath}
\mu(\OO(a,b))=a+b<-\frac{1}{2}=\mu(E).
\end{displaymath}
We have
\begin{displaymath}
0 \to \OO(-(a+1),-b) \to E(-a,-b) \to I_Z(-a,-b) \to 0,
\end{displaymath}
and
\begin{displaymath}
0 \to I_Z(-a,-b) \to \OO(-a,-b) \to \OO / I_Z \to 0.
\end{displaymath}
First note that if $a>0$ or $b>0$, then $\HH^0(\OO(-(a+1),-b))=0$ and $\HH^0(I_Z(-a,-b))=0$ since $\HH^0(\OO(-a,-b))=0$, in which case $\HH^0(E(-a,-b))=0$. We thus assume $a,b\leq 0$. But then
\begin{displaymath}
\mu(\OO(a,b)) \leq -1 < \mu(E),
\end{displaymath}
unless $a=b=0$. Moreover, if $a=b=0$, then $\HH^0(\OO(-1,0))=\HH^0(I_Z)=0$, so that $\HH^0(E)=0$, implying that there are no non-zero maps $\OO \to E$. Consequently, if $\HH^0(E(-a,-b)) \neq 0$, then $\mu(\OO(b,a)) < \mu(E)$.

Similarly, in the second case, one can show that if $\mu(\OO(a,b))= a+ b \geq 0$, then $\HH^0(E(-a,-b))=0$; otherwise we have 
\begin{displaymath}
\mu(\OO(a,b)) \leq -1 < \mu(E),
\end{displaymath}
so $E$ is in fact stable.
\end{proof}

We can now prove:

\begin{theorem}\label{nonemptinessmayorados}
The moduli space $\MM(c_1,c_2)$ is non-empty if one of the following holds:
\begin{enumerate}
\item $c_1=0,-F,-C_0$ and $c_2 \geq 0$ 
\item $c_1=-C_0-F$ and $c_2 \geq 1$
\end{enumerate} 
Morevoer, in all cases $\MM(c_1,c_2)$ contains a non-trivial semistable co-Higgs pair.
\end{theorem}

\begin{proof}
To see that $\MM(0,c_2) \neq \emptyset$, consider a rank-2 vector bundle $E$ with $c_1(E)=0$, $c_2(E) \geq 0$, and numerical invariants $d=r=0$ (such a bundle exists by Theorem \ref{ABM1}). We can write $E$ as
\begin{displaymath}
0 \to \OO \to E \to I_Z \to 0
\end{displaymath} 
with $\ell(Z)=c_2$. Since $\mu(E)=\mu(\OO)=0$, $E$ is semistable. Thus, it suffices to show that $\HH^0(\End_0 E(2,0)) \neq 0$. This follows from Proposition \ref{toshowHiggsfields}, as $\HH^0(E(2,0)) \neq 0$.

To see that $\MM(-F,c_2) \neq \emptyset$, we consider two cases. When $c_2=0$, take $E=\OO \p \OO(-1,0)$ and a Higgs field of the form
\begin{displaymath}
\Phi= \Phi_1 = \left(\begin{array}{cc}
A_1 & B_1 \\
C_1 & -A_1 \\
\end{array}\right),
\end{displaymath}
with $A_1 \in \HH^0(\OO(2,0))$, $B_1 \in \HH^0(\OO(3,0))$ and non-zero $C_1 \in \HH^0(\OO(1,0))$. Note that the only destabilizing sub-line bundle of $E$ is $\OO$. It follows that $(E,\Phi)$ is stable, as $C_1 \neq 0$, and so $\OO$ is not $\Phi$-invariant. Thus, $(E,\Phi) \in \MM(-F,0)$.

When $c_2 \geq 1$, consider a rank-2 vector bundle $E$ with $c_1(E)=-F$, and numerical invariants $d=0$, $r=-1$ (such a bundle exists by Theorem \ref{ABM1}). We can write $E$ as
\begin{displaymath}
0 \to \OO(-1,0) \to E \to I_Z \to 0
\end{displaymath} 
with $\ell(Z)=c_2$. We proved in Lemma \ref{stabilityfor-Fand-C_0} that such an $E$ is stable. Thus, it suffices to show that $\HH^0(\End_0 E(2,0)) \neq 0$. This follows from Proposition \ref{toshowHiggsfields}, as $\HH^0(E(2,0)) \neq 0$.

To see that $\MM(-C_0,c_2) \neq \emptyset$, we again consider two cases. When $c_2 = 0$, take $E=\OO \p \OO(0,-1)$ and a Higgs field of the form 
\begin{displaymath}
\Phi= \Phi_2 = \left(\begin{array}{cc}
A_2 & B_2 \\
C_2 & -A_2 \\
\end{array}\right),
\end{displaymath}
with $A_2 \in \HH^0(\OO(0,2))$, $B_2 \in \HH^0(\OO(0,3))$ and $C_2 \in \HH^0(\OO(0,1))$. Any non-zero $C_2$ will not leave $\OO$ $\Phi$-invariant, which is the only destabilizing sub-line bundle of $E$. Thus $(E, \Phi)$ yields a stable pair in $\MM(-C_0,0)$.

When $c_2 \geq 1$, consider a rank-2 vector bundle $E$ with $c_1(E)=-C_0$, and numerical invariants $d=0$, $r=-1$ (such a bundle exists by Theorem \ref{ABM1}). We can write $E$ as
\begin{displaymath}
0 \to \OO(-1,0) \to E \to \OO(1,-1) \tt I_Z \to 0
\end{displaymath} 
with $\ell(Z)=c_2-1$. Again, we proved in Lemma \ref{stabilityfor-Fand-C_0} that such an $E$ is stable. Thus, it suffices to show that $\HH^0(\End_0 E(2,0)) \neq 0$. This follows from Proposition \ref{toshowHiggsfields}, as $\HH^0(E(2,0)) \neq 0$.

Finally, to see that $\MM(-C_0-F,c_2) \neq \emptyset$, consider a rank-2 vector bundle $E$ with $c_1(E)=-C_0-F$, $c_2(E) \geq 1$, and numerical invariants $d=0$, $r=-1$ (such a bundle exists by Theorem \ref{ABM1}). We can write $E$ as
\begin{displaymath}
0 \to \OO(-1,0) \to E \to \OO(0,-1) \tt I_Z \to 0
\end{displaymath} 
with $\ell(Z)=c_2(E)-1$. Since $\mu(E)=\mu(\OO(-1,0))=-1$, $E$ is semistable. Thus, it suffices to show that $\HH^0(\End_0 E(2,0)) \neq 0$. This follows from Proposition \ref{toshowHiggsfields}, as $\HH^0(E(2,0)) \neq 0$.
\end{proof}

We now work with arbitrary first Chern class $c_1$, and give a complete characterization of when $\MM(c_1,c_2)$ is non-empty. 


\begin{theorem}\label{fullanswernon-emptiness}
Let $c_1=\alpha C_0 + \beta F$ and $c_2=\gamma$. Then, the moduli space $\MM(c_1,c_2)$ is non-empty (and moreover it contains a non-trivial co-Higgs pair) if and only if one of the following holds:
\begin{enumerate}
\item at least one of $\alpha$ and $\beta$ is even and $2\gamma \geq \alpha\beta$;
\item $\alpha$ and $\beta$ are both odd and $2\gamma \geq \alpha \beta -2$.
\end{enumerate}
\end{theorem}

\begin{proof}
For the forward direction, take $E \in \MM(c_1,c_2)$ and tensor it by the appropriate line bundle $L$, so that its first Chern class is one of the reduced ones (see the proof of Lemma \ref{decomposable}). Then Corollary \ref{modifiedc_2} tells us how the second Chern class changes after tensoring $E$ by $L$. We can then apply Theorem \ref{extid} to obtain the desired result. The converse follows from the results in this section on the existence of semistable co-Higgs bundles with one of the reduced classes, see Theorem \ref{nonemptinessmayorados}.
\end{proof}

\begin{remark}
The main results in this section can be generalized to Hirzebruch surfaces, and when the second Chern class is smaller than two, to arbitrary polarizations. Even though the arguments to prove these results are very similar, there are some subtleties that have to be taken into account; for instance, the results for the reduced classes $-C_0$ and $-F$ have a different flavour. We do not pursue this here, as the main focus of this paper is to explore semistable co-Higgs bundles only over $\P \times \P$. However, these generalizations can be found in my PhD thesis \cite{AVC} and will also appear on a future paper.
\end{remark}

We finish this section by observing that not every stable bundle $E$ admits a non-zero Higgs field. Indeed, we have:

\begin{proposition}\label{nonon-trivial}
Suppose that $d>1$, $r \leq -1-d$ and $c_2 \geq 3 - d(1+2r)$, or that $d=1$, $r \leq -2$ and $c_2 \geq -4r-1$, then $M(-F,c_2,d,r) \neq \emptyset$. Moreover, each $E \in M(-F,c_2,d,r)$ is stable and has no non-trivial Higgs field.
\end{proposition}

The proof of the above Proposition is lengthy and technical, so instead of including it here, we refer the reader to \cite[Chapter~3]{AVC}.

\section{The Moduli Spaces $\MM(c_1,c_2)$}\label{section5}

In this section we construct explicit examples of moduli spaces for reduced first Chern classes ($0$, $-F$, $-C_0$, $-C_0-F$), and low values of the second Chern class. More specifically, we give a full description of the moduli spaces of rank 2 semistable co-Higgs bundles for $c_2=0$ (and any of the reduced classes for $c_1$). In the case of $c_2=1$, we also give an example (when $c_1=-F$) of how the moduli space $\MM(-F,1)$ looks like. In this case, a technical obstacle is obtaining the Higgs fields for non-trivial extensions of a line bundle by another line bundle (which are not decomposable). Though the idea of how to achieve this is straightforward, the execution is computationally heavy. 

\subsection{Second Chern Class $c_2=0$}

Throughout this subsection, we let $E$ denote a rank 2 vector bundle over $\P \times \P$ and assume $c_2=0$. We also assume that $c_1$ is either $0$ or $-F$, since the case where $c_1=-C_0$ is symmetric to the case $c_1=-F$ (in the sense that one simply interchanges the roles of the first and second copies of $\P$). Recall that, in this case, by Theorem \ref{extid}, any semistable co-Higgs  pair $(E,\Phi)$ is such that $E$ is an extension of line bundles. Moreover, we can prove that, in this case, $E$ is decomposable (Propositions \ref{budleandhiggsfields-F0} and \ref{underlyingthreedecomposable}).

\subsubsection{First Chern class $c_1=-F$}\label{sectionmodulispace-F0}

We now analyze further the case $c_1=-F$. Recall that in this case the notions of semistability and stability coincide. We begin by describing the possible co-Higgs pairs appearing in the moduli space. 

\begin{proposition}\label{budleandhiggsfields-F0}
Suppose that $c_1(E)=-F$ and $c_2(E)=0$. If $(E,\Phi)$ is a stable co-Higgs pair, then $E=\OO \p \OO(-1,0)$. Moreover, $\Phi$ is of the form
\begin{displaymath}
\Phi=\Phi_1+\Phi_2=\left(\begin{array}{cc}
A_1 & B_1 \\
C_1 & -A_1 \\
\end{array}\right)
+
\left(\begin{array}{cc}
A_2 & B_2 \\
0 & -A_2 \\
\end{array}\right),
\end{displaymath}
with 
$A_1 \in \HH^0(\OO(2,0))$, $B_1 \in \HH^0(\OO(3,0))$, $C_1 \in \HH^0(\OO(1,0))$  and $A_2 \in \HH^0(\OO(0,2))$, $B_2 \in \HH^0(\OO(1,2))$.
\end{proposition}

\begin{proof}
By Theorem \ref{extid}, $E$ is an extension of line bundles. Let us first show that $E$ is in fact decomposable. We know that $E$ fits into an exact sequence of the form
\begin{displaymath} 
0 \rightarrow \mathcal{O}(a,b) \rightarrow E \rightarrow \mathcal{O}(-1-a,-b) \rightarrow 0, 
\end{displaymath}
and $0=c_2(E)=-b(1+2a)$. Thus $b=0$. A non-trivial extension corresponds to an element of  $\HH^1(\OO(2a+1,0)) = \HH^1(\P; \OOO(2a+1))$, and so it is the pullback to $\mathbb{P}^1 \times \mathbb{P}^1$ of a non-trivial extension $V$ of $\OOO(-a-1)$ by $\OOO(a)$ over $\P$. Since bundles over $\P$ are decomposable, $V = \OOO(c) \oplus \OOO(c')$, for some integers $c$ and $c'$, and $E = \pr_1^*V = \OO(c,0) \oplus \OO(c',0)$. 

Then, the underlying bundle of a stable co-Higgs pair $(E,\Phi)$ with $c_1(E)=-F$ and $c_2(E)=0$ is of the form $E = \OO(a,0) \p \OO(-a-1,0)$. Let us show that $a$ can only take the values $-1$ or $0$. Any element  $\Phi \in \HH^0(\End_0E \tt T)$ has the form
\begin{displaymath}
\Phi=\left(\begin{array}{cc}
A_1 & B_1 \\
C_1 & -A_1 \\
\end{array}\right)
+ 
\left(\begin{array}{cc}
A_2 & B_2 \\
C_2 & -A_2 \\
\end{array}\right), 
\end{displaymath}
where $A_1 \in $ H$^0(\OO(2,0))$, $B_1 \in $ H$^0(\OO(2a+3,0))$, $C_1 \in \HH^0(\OO(1-2a,0))$, and $A_2 \in $ H$^0(\OO(0,2))$, $B_2 \in $ H$^0(\OO(2a+1,2))$, $C_2 \in $ H$^0(\OO(-1-2a,2))$.
If $a\geq 1$, then $C_1= C_2=0$, and $\OO(a,0)$ is $\Phi$-invariant. However, $\mu(\OO(a,0)) > \mu(E)$, which contradicts stability. A similar argument, but interchanging the roles of the $C_i$'s for the $B_i$'s, and of $\OO(a,0)$ for $\OO(-1-a,0)$, shows that $a > -2$. Hence, $a=0,-1$ and thus $E=\OO \p \OO(-1,0)$.

Let us now determine which $\Phi$'s yield stable pairs $(E, \Phi)$. Any element in $\HH^0(\End_0E\tt T)$ is of the form:
\begin{displaymath}
\Phi=\Phi_1+\Phi_2=\left(\begin{array}{cc}
A_1 & B_1 \\
C_1 & -A_1 \\
\end{array}\right)
+
\left(\begin{array}{cc}
A_2 & B_2 \\
0 & -A_2 \\
\end{array}\right),
\end{displaymath}
with 
$A_1 \in $ H$^0(\OO(2,0))$, $B_1 \in $ H$^0(\OO(3,0))$, $C_1 \in $ H$^0(\OO(1,0))$  and $A_2 \in $ H$^0(\OO(0,2))$, $B_2 \in $ H$^0(\OO(1,2))$. Note that, if $\Phi$ were a Higgs field of $E$, then $C_1$ must be non-zero, as otherwise it would leave $\OO$ invariant, contradicting stability. Also, taking into account the integrability condition, equations (\ref{integrability}) imply that $A_2=B_2=0$. Therefore, any possible Higgs field of $E$ is of the form $\Phi=\Phi_1 \in \HH^0(\End_0 E \tt \OO(2,0))$, with $\Phi_1$ as above, and non-zero $C_1$. Now, the fact that $(E,\Phi)$ is indeed stable for any of these Higgs fields follows from Lemma \ref{indeedsemistable}.
\end{proof}

Given Proposition \ref{budleandhiggsfields-F0}, we now discuss the isomorphism classes of pairs $(\OO \p \OO(-1,0), \Phi)$ with $\Phi$ as above. Recall that $(E, \Phi)$ is isomorphic to $(E,\Phi')$ when there exists an automorphism $\Psi$ of $E$ such that $\Phi'=\Psi\circ\Phi\circ\Psi^{-1}$. Now, an automorphism $\Psi$ of $E$ can be chosen of the form $$\Psi=\left(\begin{array}{cc} 1 & P \\ 0 & Q \end{array}\right) \in \HH^0(\End E),$$ where $P$ and $Q$ are global sections of $\OO(1,0)$ and $\OO$, respectively; moreover, $Q \neq 0$. Hence, 
\begin{displaymath}
\Psi \Phi \Psi^{-1}=\left(
\begin{array}{cc}
A_1+PC_1 & -Q^{-1}(2A_1P-B_1-C_1P) \\
QC_1 & -(A_1+PC_1)
\end{array}\right).
\end{displaymath}
Since $C_1 \in \HH^0(\OO(1,0))$, we can locally write $C_1=\alpha(z_1-p)$. Then, $A_1=A_1(p)+(z_1-p)[A_1'(p)+A_1''(p)(z_1-p)]$. It is not hard to see that, by choosing $P=-\alpha^{-1}[A_1'(p)+A_1''(p)(z_1-p)]$ and $Q=\alpha^{-1}$, we have a representative of the conjugacy class of $\Phi$ of the form
\begin{displaymath}
\left(
\begin{array}{cc}
A_1(p) & B'_1 \\
z_1-p & -A_1(p)
\end{array}\right),
\end{displaymath}
where $B'_1 \in \HH^0(\OO(3,0))$.

It follows from the above discussion that every Higgs field of a stable co-Higgs pair is the pullback of a Higgs field of the bundle $\OOO \p \OOO(-1)$ over $\P$. Furthermore, every stable co-Higgs  pair of degree $-1$ over $\P$ gives rise, by taking pullbacks, to a  stable co-Higgs  pair over $\P \times \P$ of this form. Let us state these facts as:

\begin{theorem}\label{explicitdescriptionmodulispace0}
The moduli space $\MPP(-F,0)$ of rank 2 stable co-Higgs bundles over $\P \times \P$ with first Chern class $-F$ and second Chern class $0$ is a 6-dimensional smooth variety isomorphic to the moduli space $\MPdeg(-1)$ of rank 2 stable co-Higgs bundles of degree $-1$ over $\P$.
\end{theorem}

\begin{proof}
In \cite{SR1}, Rayan proved that $\MPdeg(-1)$ is a 6-dimensional smooth variety given by 
\begin{displaymath}
\mathcal{V}:=\{(y,\rho) \in \Tot(\OOO(2)) \times \HH^0(\P;\OOO(4)) : \eta^2=\rho(\pi(y))\},
\end{displaymath}
where $\pi: \Tot(\OOO(2)) \to \P$ is the natural projection, and $\eta$ is the tautological section of the pullback of $\OOO(2)$ to its own total space.

Now, consider the map 
\begin{displaymath}\begin{array}{cccc}
f: & \MPdeg(-1) & \to & \MPP(-F,0) \\
   & (V,\varphi) & \mapsto & ((\pr_1)^*V,(\pr_1)^*\varphi).
\end{array}
\end{displaymath}
We only check that this map is well-defined; the fact that it is an isomorphism is immediate. Since the only underlying bundle of a stable co-Higgs  pair living in $\MPdeg(-1)$ is $\OOO(-1) \p \OOO$, and since $(\pr_1)^*(\OOO(-1) \p \OOO)=\OO(-1,0) \p \OO$, this map is simply given by $f(\varphi)=(\pr_1)^*(\varphi)$. Suppose $\varphi'=\psi\circ\varphi\circ\psi^{-1}$, where $\psi \in \Aut(\OOO(-1) \p \OOO)$. Then
 $$f(\varphi')=f(\psi\circ\varphi\circ\psi^{-1})=(\pr_1)^*(\psi\circ\varphi\circ\psi^{-1})=(\pr_1)^*(\psi)\circ(\pr_1)^*(\varphi)\circ(\pr_1)^*(\psi)^{-1},$$
where $(\pr_1)^*(\psi) \in \Aut((\pr_1)^*(\OOO(-1) \p \OOO))$. Thus $f(\varphi')$ and $f(\varphi)$ are indeed in the same conjugacy class.
\end{proof}

\subsubsection{First Chern class $c_1=0$}\label{sectionmodulispace00}

We now focus on the case $c_1=0$. In this case, we no longer have a single underlying bundle; nonetheless, there are only three possibilities.

\begin{proposition}\label{underlyingthreedecomposable}
Suppose that $c_1(E)=0$ and $c_2(E)=0$. If $(E,\Phi)$ is a semistable co-Higgs pair, then $E=\OO \p \OO$ or $\OO(0,1) \p \OO(0,-1)$ or $\OO(1,0) \p \OO(-1,0)$.
\end{proposition}

\begin{proof}
It follows from Theorem \ref{extid} that $E$ is an extension of line bundles. As in the proof of Propostion \ref{budleandhiggsfields-F0}, one can easily check that $E$ is decomposable and of the form $E = \OO(a,0) \p \OO(-a,0)$ or $E = \OO(0,b) \p \OO(0,-b)$. Without loss of generality, we now assume that $E = \OO(a,0) \p \OO(-a,0)$, as the other case is analogous. Let us show that $a$ can only take the values $-1$, $0$ or $1$. Any Higgs field $\Phi \in \HH^0(\End_0E \tt T)$ has the form: 
\begin{displaymath}
\Phi=\left(\begin{array}{cc}
A_1 & B_1 \\
C_1 & -A_1 \\
\end{array}\right)
+ 
\left(\begin{array}{cc}
A_2 & B_2 \\
C_2 & -A_2 \\
\end{array}\right), 
\end{displaymath}
where $A_1 \in $ H$^0(\OO(2,0))$, $B_1 \in $ H$^0(\OO(2a+2,0))$, $C_1 \in \HH^0(\OO(2-2a,0))$, and $A_2 \in $ H$^0(\OO(0,2))$, $B_2 \in $ H$^0(\OO(2a,2))$, $C_2 \in $ H$^0(\OO(-2a,2))$.
If $a\geq 2$, then $C_1= C_2=0$, and so $\OO(a,0)$ would be $\Phi$-invariant. However, $\mu(\OO(a,0)) > \mu(E)$, which contradicts stability. A similar argument, but interchanging the roles of the $C_i$'s for the $B_i$'s, and of $\OO(a,0)$ for $\OO(-a,0)$, shows that $a > -2$. The result follows.
\end{proof}

We first consider $E=\OO \oplus \OO$, which is a strictly semistable bundle. It would be desirable to describe all the possible Higgs fields that $E$ admits. However, we do not yet know the shape of the Higgs fields $\Phi$ for which $(E, \Phi)$ is stable; i.e., those $\Phi$'s for which no copy of $\OO$ inside $E$ is $\Phi$-invariant. We do, however, describe those Higgs fields which are strictly semistable. We remind the reader that, in the moduli spaces $\MPP(c_1,c_2)$, two objects are identified if they are $S$-equivalent (see Section \ref{co-Higgsbundles}).

Any element of $\HH^0(\End_0 E \tt T)$ is of the form
\begin{displaymath}
\Phi=\Phi_1 + \Phi_2 =\left(\begin{array}{cc}
A_1 & B_1 \\
C_1 & -A_1 \\
\end{array}\right)
+
\left(\begin{array}{cc}
A_2 & B_2 \\
C_2 & -A_2 \\
\end{array}\right)
\end{displaymath}
with $A_1, B_1, C_1 \in \HH^0(\OO(2,0))$ and $A_2, B_2, C_2 \in \HH^0(\OO(0,2))$.

We now need to consider the $\Phi$'s as above, which are integrable; i.e., the $\Phi$'s that satisfy equations (\ref{integrability}). Furthermore, we focus on those Higgs fields which make $(E,\Phi)$ into a strictly semistable co-Higgs bundle.

A pair $(E,\Phi)$ is strictly semistable if and only if there is a $\Phi$-invariant copy of $\OO$ in $E$. This is equivalent to the existence of a non-zero $v \in \HH^0(E)$ such that $\Phi(v)= v \tt \lambda_1 + v \tt \lambda_2$, where $\lambda_1 \in \HH^0(\OO(2,0))$ and $\lambda_2 \in \HH^0(\OO(0,2))$. Writing $\Phi$ as
\begin{displaymath}
\Phi=(M_0 + M_1z_1 +M_2z_1^2)+(N_0+N_1z_2+N_2z_2^2),
\end{displaymath}
where the $M_i$'s and the $N_i$'s are $2 \times 2$ complex valued matrices, we see that $\Phi(v)=v \tt \lambda_1 + v \tt \lambda_2$ if and only if $v$ is a common eigenvector of the $M_i$'s and $N_i$'s (note that, in this case, the coefficients of $\lambda_1, \lambda_2$ are the eigenvalues of the $M_i$'s and the $N_i$'s, respectively). If this is the case, by a change of basis, we may assume that the $M_i$'s and $N_i$'s are upper triangular. Therefore, $(E,\Phi)$ is strictly semistable if and only if $\Phi$ is upper triangular and its matrix coefficients admit a common eigenvector. In this case, by Lemma \ref{gr}, we have that
\begin{displaymath}
\gr(E,\Phi)=\left(E, \left(\begin{array}{cc}
A_1 & 0 \\
0 & -A_1 \\
\end{array}\right)
+
\left(\begin{array}{cc}
A_2 & 0 \\
0 & -A_2 \\
\end{array}\right) \right).
\end{displaymath}
Note that two graded objects $\gr(E,\Phi)$ and $\gr(E,\Phi')$ are isomorphic if and only if $A_1+A_2= \pm (A_1'+A_2')$. Indeed, two matrices $\left(\begin{array}{cc}
A & 0 \\
0 & -A \\
\end{array}\right)$ and $\left(\begin{array}{cc}
B & 0 \\
0 & -B \\
\end{array}\right)$ live in the same S-equivalence class, that is
\begin{displaymath}
\left(\begin{array}{cc}
A & 0 \\
0 & -A \\
\end{array}\right)= \Psi \left(\begin{array}{cc}
B & 0 \\
0 & -B \\
\end{array}\right) \Psi^{-1},
\end{displaymath}
if and only if $A= \pm B$. Therefore, a set of representatives for the S-equivalence classes of strictly semistable pairs, with underlying bundle $E$, is given by
\begin{displaymath}
\left \{ \left(\begin{array}{cc}
A & 0 \\
0 & -A \\
\end{array}\right) \; : \; A \in \HH^0(T) \right \} / \sim
\end{displaymath}
where $\sim$ is defined by $A \sim B$ if and only if $A = \pm B$.

On the other hand, note that for the bundle $E=\OO(1,0)\oplus\OO(-1,0)$, any element of $\HH^0(\End_0 E \tt T)$ is of the form 
\begin{displaymath}
\Phi=\Phi_1 + \Phi_2 =\left(\begin{array}{cc}
A_1 & B_1 \\
C_1 & -A_1 \\
\end{array}\right)
+
\left(\begin{array}{cc}
A_2 & B_2 \\
0 & -A_2 \\
\end{array}\right)
\end{displaymath}
with $A_1 \in \HH^0(\OO(2,0))$, $B_1 \in \HH^0(\OO(4,0))$, $C_1 \in \HH^0(\OO)$ and $A_2 \in \HH^0(\OO(0,2))$ and $B_2 \in \HH^0(\OO(2,2))$. Note that, if $\Phi$ were a Higgs field of $E$, then $C_1$ must be non-zero, as otherwise it would leave $\OO(1,0)$ invariant, contradicting stability. Also, taking into account the integrability condition, equations (\ref{integrability}) imply that $A_2=B_2=0$. Therefore, any possible Higgs field of $E$ is of the form $\Phi=\Phi_1 \in \HH^0(\End_0 E \tt \OO(2,0))$, with $\Phi_1$ as above, and non-zero $C_1$. 

Now, we observe that $(E,\Phi)$ is in fact co-Higgs stable for any $\Phi$ as above. Note that the sub-line bundles of $E$ are of the form $\OO(r,s)$ with $r \leq 1$ and $s \leq 0$ or $r \leq -1$ and $s \leq 0$. As such, the only sub-line bundles that could potentially contradict stability are $\OO$ and $\OO(1,0)$. However, any degree zero sub-line bundle of $E$; that is, any copy of $\OO$ in $E$, is contained in $\OO(1,0)$; and so, since the latter is not $\Phi$-invariant, the result follows. 

We now claim that a set of representatives for the isomorphism classes (recall that in this case S-equivalence reduces to co-Higgs isomorphism) is given by
\begin{displaymath}
\left \{ \left(\begin{array}{cc}
0 & B \\
1 & 0 \\
\end{array}\right) \; : \; B \in \HH^0(\OO(4,0)) \right \}.
\end{displaymath}

Indeed, given any Higgs field $\Phi$ as described above, letting 
\begin{displaymath}
\Psi=\left(\begin{array}{cc}
1 & -A_1 \\
0 & 1 \\
\end{array}\right)
\end{displaymath}
we get
\begin{displaymath}
\Psi \Phi \Psi^{-1}= \left(\begin{array}{cc}
0 & B_1-A_1^2 \\
1 & 0 \\
\end{array}\right).
\end{displaymath}

Similarly, for the bundle $E=\OO(0,1)\oplus\OO(0,-1)$, a set of representatives for the isomorphism class is
\begin{displaymath}
\left \{ \left(\begin{array}{cc}
0 & B \\
1 & 0 \\
\end{array}\right) \; : \; B \in \HH^0(\OO(0,4)) \right \}.
\end{displaymath}

Following the description of \cite[Section~8]{SR2} for the case of curves, we now show that we can view the set of stable co-Higgs bundles, in $\MPP(0,0)$, with underlying bundle $\OO(1,0) \p \OO(-1,0)$ or $\OO(0,1) \p \OO(0,-1)$ as sections of certain maps, as follows:

Consider the maps
\begin{displaymath}
\begin{array}{cccc}
h_1: & \MPP(0,0) & \to & \HH^0(\OO(4,0)) \\
        & (E, \Phi=\Phi_1+\Phi_2) & \mapsto & \det \Phi_1
\end{array}
\end{displaymath}
and
\begin{displaymath}
\begin{array}{cccc}
h_2: & \MPP(0,0) & \to & \HH^0(\OO(0,4)) \\
        & (E, \Phi=\Phi_1+\Phi_2) & \mapsto & \det \Phi_2.
\end{array}
\end{displaymath}
Define
\begin{displaymath}
\begin{array}{cccc}
Q_1: & \HH^0(\OO(4,0)) & \to & \MPP(0,0) \\
        & \rho & \mapsto & \left(\OO(1,0) \p \OO(-1,0), \Phi_1=\left(\begin{array}{cc}
0 & -\rho \\
1 & 0 \\
\end{array}\right)\right)
\end{array}
\end{displaymath}
and
\begin{displaymath}
\begin{array}{cccc}
Q_2: & \HH^0(\OO(0,4)) & \to & \MPP(0,0) \\
        & \rho & \mapsto & \left(\OO(0,1) \p \OO(0,-1), \Phi_2=\left(\begin{array}{cc}
0 & -\rho \\
1 & 0 \\
\end{array}\right)\right).
\end{array}
\end{displaymath}

Clearly, $Q_i$ is a section of $h_i$ for $i=1,2$. Moreover, by the above discussion, we have:

\begin{proposition}\label{modulispace00}
The images of the sections $Q_1, Q_2$ in $\MPP(0,0)$ are precisely the set of stable co-Higgs bundles with underlying bundle $\OO(1,0) \p \OO(-1,0)$ and $\OO(0,1) \p \OO(0,-1)$, respectively.
\end{proposition}

\begin{remark}
Note that every point in $\MPP(0,0)$ with underlying bundle $\OO(1,0) \p \OO(-1,0)$ or $\OO(0,1) \p \OO(0,-1)$ is the pullback from $\P$ to $\P \times \P$ of a stable co-Higgs pair with underlying bundle $\OOO(1) \p \OOO(-1)$ with respect to the first or second projections, respectively. Similarly, any point in $\MPP(0,0)$ with underlying bundle $\OO \p \OO$ can be obtained from a stable co-Higgs bundle $(\OOO \p \OOO, \phi)$ over $\P$ by taking the pullback of $\phi$ with respect to the two projections; i.e., $(\OO \p \OO, \pr_1^*\phi + \pr_2^*\phi)$.
\end{remark}

\subsection{Second Chern Class $c_2=1$}\label{msc2=1}

We now turn our attention to co-Higgs pairs $(E, \Phi)$ over $\P \times \P$ with $c_1=-F$ and $c_2=1$. Once again, recall that, in this case, stability and semistability are identical notions. In studying the possible underlying bundles for stable co-Higgs  pairs $(E,\Phi)$, we will show that $E$ is an extension of $\OO(-1,1)$ by $\OO(0,-1)$.

We begin by proving a technical lemma. 

\begin{lemma}\label{r=0}
Let $x$ be a point in $\P \times \P$ and suppose that $c_2(E)=1$. If $E$ fits into the exact sequence
\begin{equation}\label{ess0}
0 \to  \OO \xrightarrow{\iota} E \xrightarrow{p} \OO(-1,0) \tt I_x \to 0,
\end{equation}
then $\OO$ is $\Phi$-invariant for any $\Phi \in \HH^0(\End_0 E \tt T)$.
\end{lemma}

\begin{proof}
First note that since $\Ext(\OO(-1,0) \tt I_x, \OO) \cong \HH^0(\OO_x) \cong \C$, up to isomorphism, there is a unique $E$ that fits into (\ref{ess0}). Now, tensoring (\ref{ess0}) with $T$ and passing to the long exact sequence in cohomology, we get
\begin{equation*}\label{es00}
0 \to \HH^0(T) \xrightarrow{\iota} \HH^0(E \tt T) \xrightarrow{p} \HH^0(\OO(-1,0) \tt I_x \tt T) \to 0,
\end{equation*}
where we also denote by $\iota$ and $p$ the induced map $\iota \tt \Id_T$ and $p \tt \Id_T$, respectively. Note that $\I(\iota) = \K(p)$. Moreover, using Grothendieck--Riemann--Roch \cite[Appendix~A]{H}, one can check that $\pi_*(I_x) = \OOO(-1)$. Therefore,
\begin{displaymath}
\begin{array}{ccl}
\HH^0(\OO(-1,0) \tt I_x \tt T) & = & \HH^0(\P \times \P; I_x(1,0)) \\
                                           & = & \HH^0(\P; \pi_*(I_x(1,0))) \\
                                           & = & \HH^0(\P; \OOO(1) \tt \pi_*(I_x)) \\
                                           & = & \HH^0(\P;\OOO),
\end{array}
\end{displaymath}
so the non-zero elements of $\HH^0(\OO(-1,0) \tt I_x \tt T)$ are nowhere vanishing.

With this in mind, we now claim that if 
$$S=\{\varphi \in \HH^0(\Hom(\OO,E \tt T)\; |\; \varphi=\Phi_{|\OO}, \Phi \in \HH^0(\End_0 E \tt T) \},$$ 
$S \subseteq \I(\iota)$, implying that $\OO$ is $\Phi$-invariant for any $\Phi \in \HH^0(\End_0E\tt T)$, by Lemma \ref{isb}. Indeed, let $\varphi \in S$ so that $\varphi=\Phi \circ \iota$ for some $\Phi \in \HH^0(\End_0 E \tt T)$. Then, 
$$p(\varphi)=p\circ \Phi \circ \iota,$$
which vanishes at $x$ since $\iota$ does (otherwise the quotient $E/\iota(\OO)$ would be locally free). Hence, $p(\varphi)=0$ because non-zero elements of $\HH^0(\OO(-1,0)) \tt I_x \tt T$ are nowhere vanishing, proving that $\varphi \in \K(p)= \I(\iota)$.
\end{proof}


We are now ready to prove that the only possible underlying bundles for stable co-Higgs  pairs are extensions of $\OO(-1,1)$ by $\OO(0,-1)$.

\begin{proposition}\label{allunderlying1}
Let $(E,\Phi)$ be a stable co-Higgs pair such that $c_1(E)=-F$ and $c_2(E)=1$. Then $E$ is an extension of $\OO(-1,1)$ by $\OO(0,-1)$.
\end{proposition}

\begin{proof}
Let $E$ have invariants $d$ and $r$. Then we know that $E$ fits into an exact sequence of the form
\begin{equation}\label{es1}
0 \to \OO(r,d) \to E \to \OO(-1-r,-d) \tt I_Z \to 0,
\end{equation}
with $\ell(Z)=1+d(2r+1)\geq 0$. Now, for such a rank-2 vector bundle to exist, we know by Theorem \ref{ABM1} that one of the following two conditions must be satisfied.
\begin{enumerate}
\item $d \geq 1$, or
\item $d=0$ and $r \geq -1$.
\end{enumerate}
In case (1), we consider two subcases: 

(i) $r \geq 0$. By tensoring (\ref{es1}) with $\OO(r,d)^{\lor} \tt T$, and passing to the long exact sequence in cohomology, we get
\begin{equation*}\label{es2}
0 \to \HH^0(T) \to \HH^0(E(-r,-d) \tt T) \to \HH^0(\OO(-1-2r,-2d) \tt T \tt I_Z) \to 0,
\end{equation*}
where $\HH^0(\OO(-1-2r,-2d) \tt T \tt I_Z)=0$, and so, by Lemma \ref{isb}, $\OO(r,d)$ is $\Phi$-invariant for any $\Phi \in \HH^0(\End_0 E \tt T)$, and thus destabilizing. This contradicts stability, and thus this case cannot happen.

(ii) $r \leq -1$. If either $r<-1$, or $d>1$ and $r=-1$, then $\ell(Z)<0$, which is impossible. Hence $d=1$ and $r=-1$, so that $\ell(Z)=0$, and therefore $E$ is an extension of $\OO(0,-1)$ by $\OO(-1,1)$. However, $\HH^1(\OO(-1,2))=0$, implying that $E=\OO(-1,1) \p \OO(0,-1)$.

\noindent
In case (2) we consider three subcases: 

(i) $r \geq 1$. By the exact same argument as above, one can check that $\OO(r,d)$ is $\Phi$-invariant for any $\Phi \in \HH^0(\End_0 E \tt T)$, and thus destabilizing. Hence this case cannot happen.

(ii) $r=0$. In this case, $E$ fits into an exact sequence of the form
\begin{equation*}
0 \to \OO \to E \to \OO(-1,0) \tt I_x \to 0,
\end{equation*}
and so, by Lemma \ref{r=0}, $\OO$ is $\Phi$-invariant for any $\Phi \in \HH^0(\End E \tt T)$, and thus destabilizing. Again, this case cannot happen.

(iii) $r=-1$. In this case, $E$ fits into an exact sequence of the form
\begin{equation*}
0 \to \OO(-1,0) \to E \to I_x \to 0.
\end{equation*}
Now, since $\Ext(I_x, \OO(-1,0)) = \HH^0(\OO_x)=\C$, there is a unique bundle, up to isomorphism, that fits into this exact sequence. Hence, $E$ is completely determined by the invariants $d=0$ and $r=-1$, up to isomorphism. On the other hand, any non-trivial extension $E'$ of $\OO(-1,1)$ by $\OO(0,-1)$ has invariants $d=0$ and $r=-1$. Indeed, the restriction of $E'$ to the generic fibre is a non-trivial extension of $\OOO(1)$ by $\OOO(-1)$ over $\P$, so $d=0$. Moreover, pushing down the extension
$$0 \to \OO(-1,1) \to E' \to \OO(0,-1) \to 0$$
to $\P$, we obtain
$$(\pr_1)_*(E)=(\pr_1)_*(\OO(-1,1))=\OOO(-1) \p \OOO(-1).$$
Thus, $r=-1$. Hence, $E \cong E'$ and $E$ is a non-trivial extension of $\OO(-1,1)$ by $\OO(0,-1)$.
\end{proof}

As before, now that we know the possible underlying bundles for stable co-Higgs pairs, we can check whether they admit (non-trivial) Higgs fields. 

We start by working with the trivial extension $E=\OO(0,-1) \p \OO(-1,1)$. In this case, any element of $\HH^0(\End_0E \tt T)$ is of the form
\begin{displaymath}
\Phi=\Phi_1 + \Phi_2 =\left(\begin{array}{cc}
A_1 & 0 \\
C_1 & -A_1 \\
\end{array}\right)
+
\left(\begin{array}{cc}
A_2 & B_2 \\
0 & -A_2 \\
\end{array}\right),
\end{displaymath}
with $A_1 \in \HH^0(\OO(2,0))$, $C_1 \in \HH^0(\OO(1,2))$, and $A_2 \in \HH^0(\OO(0,2))$, $B_2 \in \HH^0(\OO(1,0))$. Note that $B_2$ cannot be identically zero, for otherwise it would leave $\OO(-1,1)$ invariant, contradicting stability. Again, taking into account the integrability condition, equations (\ref{integrability}) imply that $A_1=C_1=0$. Therefore, any possible Higgs field of $E$ is of the form 
$$\Phi= \Phi_2 \in \HH^0(\End_0(0,2)),$$
with $\Phi_2$ as above and $B_2$ not identically zero. The fact that $(E,\Phi)$ is indeed co-Higgs stable for these Higgs fields follows from Lemma \ref{indeedsemistable}. 


Now, note that an automorphism $\psi$ of $E=\OO(0,-1) \p \OO(-1,1)$ can be chosen of the form
\begin{displaymath}
\psi=\left(\begin{array}{cc} 1 & 0 \\ 0 & P \end{array}\right) \in \HH^0(\End E), 
\end{displaymath}
where $P$ is a non-zero global section of $\OO$. We then have that 
\begin{displaymath}
\psi \Phi \psi^{-1}=\left(\begin{array}{cc}
A_2 & P^{-1}B_2 \\
0 & -A_2 \\
\end{array}\right).
\end{displaymath}
Since $B_2 \in \HH^0(\OO(1,0))$, we can locally write $B_2=\alpha(z_1-p)$, so by choosing $P=\alpha^{-1}$, we
have a representative of the conjugacy class of $\Phi$ of the form
\begin{equation}\label{phitrivial}
\Phi=\left(
\begin{array}{cc}
A_2 & z_1-p \\
0 & -A_2
\end{array}\right).
\end{equation}

We now turn our attention to the non-trivial extensions
\begin{equation}\label{nontriv00}
0 \rightarrow \OO(0,-1) \rightarrow E \rightarrow \OO(-1,1) \rightarrow 0.
\end{equation}
These, as they are stable bundles, admit the zero Higgs field. They nonetheless also admit non-zero Higgs fields. To prove this, we do the following:
\begin{enumerate}
\item[(i)] Compute the dimension $\h^0(\End_0 E \tt T)$.
\item[(ii)] Check which elements of $\HH^0(\End_0 E \tt T)$ satisfy the integrability condition.
\end{enumerate}

A direct computation gives (i):

\begin{lemma}\label{cohonontrivialext}
If $E$ is a non-trivial extension of $\OO(-1,1)$ by $\OO(0,-1)$, then $\h^0(\End_0 E(2,0))=6$ and $\h^0(\End_0 E(0,2))=5$. In particular,
$$\h^0(\End_0 E \tt T)=11.$$
\end{lemma}

\begin{proof}
The non-trivial extension $E$ is given by a class in 
$$\HH^1(\OO(1,-2))=\HH^0(\P,\OOO(1)) \p \HH^1(\P, \OOO(-2)),$$ 
which vanishes at a single point $x_0$ in the first factor of $\P \times \P$. Therefore,
\begin{displaymath}
E_{|_{F_x}} = 
\left\{ \begin{array}{ll}
\OOO \p \OOO & \text{ if  } x \neq x_0   \\
\OOO(-1) \p \OOO(1) & \text{ if  } x = x_0
\end{array} \right.
\end{displaymath}
where $F_x = (\pr_1)^{-1}(x)$.

Now, in order to compute the dimension of $\HH^0(\End_0E(2,0))$, take the dual sequence of (\ref{nontriv00}), tensor it by $E \tt \OO(2,0)$, and then push-forward it to the first copy of $\P$, we are left with:
\begin{displaymath}
0 \rightarrow (\!\pr_1\!)_*\! \End E(2,\!0) \rightarrow (\!\pr_1\!)_*\! E(2,\!1) \rightarrow R^1\!(\!\pr_1\!)_*\!E(3,\!-1) \rightarrow R^1\!(\!\pr_1\!)_*\!\End E(2,\!0) \rightarrow 0.
\end{displaymath}

Note that both $R^1(\pr_1)_*E(3,-1)$ and $R^1(\pr_1)_*\End E(2,0)$ are skyscraper sheaves supported at $x_0$. Hence, $(\pr_1)_*\End E(2,0) \cong (\pr_1)_*E(2,1)$, and so we get that $\HH^0(\End E(2,0)) = \HH^0(E(2,1))$. Now, tensoring (\ref{nontriv00}) by $\OO(2,1)$ and passing to the long exact sequence in cohomology, we get
\begin{displaymath}
0 \rightarrow \HH^0(\OO(2,0)) \rightarrow \HH^0(E(2,1)) \rightarrow \HH^0(\OO(1,2)) \rightarrow 0.
\end{displaymath}

Hence, $\h^0(E(2,1))=\h^0(\OO(2,0)) + \h^0(\OO(1,2)) = 9$, and so $\h^0(\End E(2,0))=9$. Finally, since $\End E(2,0) = \End_0 E(2,0) \p \OO(2,0)$, we get that $\h^0(\End_0E(2,0))=6$.

Now, in order to compute the dimension of $\HH^0(\End_0 E(0,2))$, take the dual sequence of (\ref{nontriv00}) and tensor it by $E(0,2)$ to get:
\begin{equation}\label{nontriv08}
0 \rightarrow \HH^0(E(1,1)) \rightarrow \HH^0(\End E(0,2)) \rightarrow \HH^0(E(0,3)) \rightarrow \HH^1(E(1,1)) \rightarrow \cdots
\end{equation}

In order to compute $\h^i(E(1,1))$, tensor (\ref{nontriv00}) by $\OO(1,1)$ and pass to the long exact sequence in cohomology:
\begin{displaymath}
0 \rightarrow \HH^0(\OO(1,0)) \rightarrow \HH^0(E(1,1)) \rightarrow \HH^0(\OO(0,2)) \rightarrow 0.
\end{displaymath}
Hence $\h^0(E(1,1))=\h^0(\OO(1,0)) + \h^0(\OO(0,2))= 5$ and $\h^1(E(1,1))=0$.

Now to compute $\h^0(E(0,3))$, tensor (\ref{nontriv00}) by $\OO(0,3)$ and pass to the long exact sequence in cohomology
\begin{displaymath}
0 \rightarrow \HH^0(\OO(0,2)) \rightarrow \HH^0(E(0,3)) \rightarrow 0
\end{displaymath}
Hence $\h^0(E(0,3))=\h^0(\OO(0,2))=3$. It now follows, from (\ref{nontriv08}), that $$\h^0(\End E(0,2))=\h^0(E(1,1)) + \h^0(E(0,3)) = 8.$$ Finally, we have that $\h^0(\End_0E(0,2))=\h^0(\End E(0,2)) - \h^0(\OO(0,2))=5$. Therefore, 
$$\h^0(\End_0E\tt T)=\h^0(\End_0E(2,0))+\h^0(\End_0(0,2))=11.$$
\end{proof}

Let us now determine which elements of $\HH^0(\End_0 E \tt T)$ satisfy the integrability condition. We begin by giving a local description of $\HH^0(\End_0E \tt T)$. In order to do so, we will need to compute the transition functions of $\End_0E \tt T$, and so we will work locally. We fix the standard open cover of $\P \times \P$:
\begin{displaymath}\begin{array}{l}
V_1=\U_0^1 \times \U_0^2 \\
V_2=\U_0^1 \times \U_{\infty}^2 \\
V_3=\U_{\infty}^1 \times \U_0^2 \\
V_4=\U_{\infty}^1 \times \U_{\infty}^2,
\end{array}
\end{displaymath}
where $\U_0^i$ is the affine open subset of the $i$-th copy of $\P$ that does not contain the point at infinity, and $\U_{\infty}^i$ is the affine open subset of the $i$-th copy of $\P$ that does not contain zero. Let us work on the intersection $V_{ij}:=V_i \cap V_j$. If we let $(uz_1+v)z_2^{-1}$ be the (non-zero) element in $\HH^1(\OO(1,-2))$ that determines the (non-trivial) extension $E$, then we know that in $V_{12}$ and $V_{13}$ the transition functions of $E$ are given by
\begin{displaymath}
g_{12}^E= \left(\begin{array}{cc} 
z_2^{-1} & (uz_1+v) \\
0 &  z_2
\end{array}\right),
\end{displaymath}
and
\begin{displaymath}
g_{13}^{E}= \left(\begin{array}{cc} 
1 & 0  \\
0 & z_1^{-1}
\end{array}\right).
\end{displaymath}
Thus, letting $g_{ij}^{(2,0)}$ and $g_{ij}^{(0,2)}$ denote the transition functions of $\End_0E(2,0)$ and $\End_0E(0,2)$, respectively, we have that
\begin{displaymath}
g_{12}^{(2,0)}= \left(\begin{array}{ccc} 
1 & 0 & uz_1z_2+vz_2 \\
-2(uz_1z_2^{-1}+vz_2^{-1}) & z_2^{-2} & -(uz_1+v)^2 \\
0 & 0 & z_2^2
\end{array}\right)
\end{displaymath}

\begin{displaymath}
g_{13}^{(2,0)}= \left(\begin{array}{ccc} 
z_1^2 & 0 & 0 \\
0 & z_1^{3} & 0 \\
0 & 0 & z_1
\end{array}\right)
\end{displaymath}

\begin{displaymath}
g_{12}^{(0,2)}= \left(\begin{array}{ccc} 
z_2^2 & 0 & uz_1z_2^3+vz_2^3 \\
-2(uz_1z_2+vz_2) & 1 & -(uz_1+v)^2z_2^2 \\
0 & 0 & z_2^4
\end{array}\right)
\end{displaymath}

\begin{displaymath}
g_{13}^{(0,2)}= \left(\begin{array}{ccc} 
1 & 0 & 0 \\
0 & z_1 & 0 \\
0 & 0 & z_1^{-1}
\end{array}\right).
\end{displaymath}
Note that the above transition functions are $3 \times 3$ matrices; thus, for this purpose, we will treat the trace-free sections $\Phi_1$ and $\Phi_2$ as $3 \times 1$ vectors. Let $\Phi_j^i$, $(j=1,2)$ be the trivialization of $\Phi_j$ on $V_i$. 

We will work on the open set $V_1$. In order to describe $\Phi_1^1 \in \HH^0(\End_0E(2,0))$, let 

\begin{displaymath}
\Phi_1^1=\left( \begin{array}{l}
\sum\limits_{i,j \geq 0}a_{ij}^1z_1^iz_2^j \\ 
\sum\limits_{i,j \geq 0}b_{ij}^1z_1^iz_2^j \\
\sum\limits_{i,j \geq 0}c_{ij}^1z_1^iz_2^j 
\end{array} \right).
\end{displaymath}
Using the fact that $\Phi_1^1=g_{13}^{(2,0)}\Phi_1^3$, a straightforward computation shows that  $a_{ij}^1=0$ for $i >2$, $b_{ij}^1=0$ for $i>3$ and $c_{ij}^1=0$ for $i>1$. Similarly, using the fact that $\Phi_1^1=g_{12}^{(2,0)}\Phi_1^2$ we get that $a_{ij}^1=0$ for $j >1$, $b_{ij}^1=0$ for $j>0$ and $c_{ij}^1=0$ for $j>2$. Furthermore, we get that

\begin{displaymath}
\begin{array}{c}
a_{00}^1= \frac{1}{2}vc_{01}^1  \\
a_{01}^1= vc_{02}^1 \\
a_{10}^1= \frac{u}{2}c_{01}^1 \\ 
a_{11}^1= uc_{02}^1+vc_{12}^1 \\   
a_{20}^1= \frac{u}{2}c_{11}^1 \\
a_{21}^1= uc_{12}^1 \\
b_{00}^1= -v^2c_{02}^1 \\ 
b_{10}^1= -(v^2c_{12}^1+2uvc_{02}^1) \\
b_{20}^1= -(u^2c_{02}^1+2uvc_{12}^1)  \\
b_{30}^1= -u^2c_{12}^1, \\
\end{array}
\end{displaymath}
and so 
\begin{equation}\label{phi1}
\Phi_1^1=\left(\begin{array}{rr}
A_1 & B_1 \\
C_1 & -A_1 
\end{array}\right), 
\end{equation}
where 
\begin{equation}\label{phi1entries}
\begin{array}{l}
A_1=  \frac{1}{2}vc_{01}^1 + vc_{02}^1z_2 + \frac{u}{2}c_{01}^1z_1 + (uc_{02}^1+vc_{12}^1)z_1z_2 + \frac{u}{2}c_{11}^1z_1^2 + uc_{12}^1z_1^2z_2, \\
B_1=  -v^2c_{02}^1-(v^2c_{12}^1+2uvc_{02}^1)z_1 -(u^2c_{02}^1+2uvc_{12}^1)z_1^2 -u^2c_{12}^1z_1^3, \\
C_1=c_{00}^1 + c_{01}^1z_2 + c_{02}^1z_2^2 + c_{10}^1z_1 + c_{11}^1z_1z_2 + c_{12}^1z_1z_2^2.
\end{array}
\end{equation}

\begin{remark}\label{yei}
Note that the above equations imply that, in $\Phi_1$, $A_1$ and $B_1$ depend on $C_1$. In particular, if $C_1$ is zero, then $A_1=B_1=0$ and $\Phi=0$. 
\end{remark}

We will now describe $\Phi_2^1 \in \HH^0(\End_0E(0,2))$. As before, let 
\begin{displaymath}
\Phi_2^1=\left( \begin{array}{l}
\sum\limits_{i,j \geq 0}a_{ij}^2z_1^iz_2^j \\
\sum\limits_{i,j \geq 0}b_{ij}^2z_1^iz_2^j \\
\sum\limits_{i,j \geq 0}c_{ij}^2z_1^iz_2^j 
\end{array} \right).
\end{displaymath}
Using the fact that $\Phi_2^1=g_{13}^{(0,2)}\Phi_2^3$, again, a straightforward computation shows that $a_{ij}^2=0$ for $i >0$, $b_{ij}^2=0$ for $i>1$ and $c_{ij}^2=0$ for all $i, j$. Similarly, using the fact that $\Phi_2^1=g_{12}^{(0,2)}\Phi_2^2$, we get that $a_{ij}^2=0$ for $j >2$ and $b_{ij}^2=0$ for $j>1$. Furthermore, we get that
\begin{displaymath}
\begin{array}{c}
b_{01}^2= -2va_{02}^2 \\
b_{11}^2= -2ua_{02}^2 
\end{array}
\end{displaymath}
and so 
\begin{equation}\label{phi2}
\Phi_2^1=\left(\begin{array}{rr}
A_2 & B_2 \\
0 & -A_2 
\end{array}\right), 
\end{equation}
where 
\begin{equation}\label{phi2entries}
\begin{array}{l}
A_2=  a_{00}^2 + a_{01}^2z_2 + a_{02}^2z_2^2, \\
B_2=  b_{00}^2 + b_{10}^2z_1 - 2(uz_1+v)a_{02}^2z_2.
\end{array}
\end{equation}

For the following lemma we use the notation described above. 

\begin{lemma}\label{phi1orphi2}
Let $\Phi \in \HH^0(\End_0E \tt T)$ be integrable. If $C_1=0$, then $\Phi=\Phi_2$. Otherwise, $\Phi=\Phi_1$.
\end{lemma}

\begin{proof}
It suffices to prove the lemma on the open set $V_1$. Indeed, in any other standard open set, the Higgs field is a conjugation (by the transition functions of $E$) of its trivialization on $V_1$. Recall that, since $\Phi$ is integrable, it satisfies equations (\ref{integrability}). Now, it is clear that if $C_1=0$, then $A_1=B_1=0$ (this follows simply by the shape of $A_1, B_1, C_1$, see (\ref{phi1entries})), and so $\Phi=\Phi_2$. On the other hand, if $C_1 \neq 0$, then $A_2=B_2=0$ (this follows from equations (\ref{integrability})), and so $\Phi=\Phi_2$. 
\end{proof}

We now aim to give a geometric description of the moduli space $\MPP(-F,1)$ of rank 2 stable co-Higgs bundles with first Chern class $-F$ and second Chern class~$1$. We have seen that if $(E,\Phi) \in \MPP(-F,1)$, then $E$ is an extension of $\OO(-1,1)$ by $\OO(0,-1)$. Such extensions are parametrized (up to strong isomorphism) by $\HH^1(\OO(1,-2)) = \C^2$, and have transition functions on $V_{12}$ given by
\begin{displaymath}
\left(\begin{array}{cc} 
g_{12} & uz_1+v \\
0 &  g'_{12}
\end{array}\right),
\end{displaymath}
for $(u,v) \in \C^2$. We will use the convenient notation $E=E_{u,v}$.

\begin{lemma}\label{uptowi}
Let $E$ and $E'$ be extensions of $\OO(-1,1)$ by $\OO(0,-1)$. If $E$ and $E'$ are isomorphic as vector bundles, then $E$ and $E'$ are weakly isomorphic extensions.
\end{lemma}

\begin{proof}
Let $E=E_{u,v}$ and $E'=E_{u',v'}$. Also, let $p=uz_1+v$ and $p'=u'z_1+v'$. Now, suppose $\alpha \in \P$ is a zero of $p$. We have that for each $z \in \P$, $E_{|\{z\} \times \P} \cong E'_{|\{z\} \times \P}$. Since the only extensions of $\OO(1)$ by $\OO(-1)$ over $\P$ are the split one $\OOO(-1) \p \OOO(1)$ and $\OOO \p \OOO$, we have that
\begin{displaymath}
E_{|_{\{z\} \times \P}} = 
\left\{ \begin{array}{ll}
\OOO \p \OOO & \text{ if  } z \neq \alpha   \\
\OOO(-1) \p \OOO(1) & \text{ if  } z = \alpha.
\end{array} \right.
\end{displaymath}
Since the same can be said of $E'_{|_{\{z\} \times \P}}$ and $p'$, we have that $p$ and $p'$ have exactly the same zeroes. Hence, $E$ and $E'$ are weakly isomorphic.
\end{proof}

\begin{remark}
Recall that, up to weak isomorphism, non-trivial extensions of $\OO(-1,1)$ by $\OO(0,-1)$ are parametrized by $\mathbb{P}(\HH^1(\OO(1,-2))) = \P$.
\end{remark} 

Let $X_0:=\{(E,\Phi) \in \MPP(-F,1): \; E \text{ is the trivial extension} \}$. By (\ref{phitrivial}), we have that $X_0 = \P \times \C^3$. Now, let us fix a non-trivial extension $E_{u,v}$, and let $X_{u,v}$ be the set of elements in $\MPP(-F,1)$ with underlying bundle $E_{u,v}$. By (\ref{phi1}), (\ref{phi1entries}), (\ref{phi2}), (\ref{phi2entries}) and Lemma \ref{phi1orphi2}, we get that
\begin{displaymath}
X_{u,v}=\{(\bar{x},\bar{y}) \in \C^6 \times \C^5; \; x_iy_j=0, \; 1 \leq i \leq 6 \text{ and } 1 \leq j \leq 5\}.
\end{displaymath}
Since $X_{u,v}$ is the union of the two subspaces $\bar{x}=0$ and $\bar{y}=0$, $\dim X_{u,v}=6$. It follows that:

\begin{proposition}\label{modulispace1beforequotiening}
The space $S=S_0 \cup S_1 \cup S_2$, where 
\begin{displaymath}
\begin{array}{ccl} 

\vspace{.05in}

S_0 &= &\{((u,v), (p,\bar{w}), (\bar{x}, \bar{y})) \in \C^2 \times \P \times \C^3 \times \C^6 \times \C^5: \; u=v=\bar{x}=\bar{y}=0 \}, \\

\vspace{.05in}

S_1& = &\{((u,v), (p,\bar{w}), (\bar{x}, \bar{y})) \in \C^2 \times \P \times \C^3 \times \C^6 \times \C^5: \; (u,v)\neq (0,0), p=\bar{w}=\bar{y}=0 \}, \\
S_2 & = & \{((u,v), (p,\bar{w}), (\bar{x}, \bar{y})) \in \C^2 \times \P \times \C^3 \times \C^6 \times \C^5: \; (u,v)\neq (0,0), p=\bar{w}=\bar{x}=0 \}, 
\end{array}
\end{displaymath}
parametrizes rank 2 stable co-Higgs bundles with first Chern class $-F$ and second Chern class 1. 
\end{proposition}

\begin{remark}
Note that $S_0 = X_0$ parametrizes the points of the form $(\OO(0,-1) \p \OO(-1,1),\Phi)$, $S_1$ the points $(E,\Phi_1)$ and $S_2$ the points $(E,\Phi_2)$, where $E$ is a non-trivial extension.
\end{remark}

By Lemma \ref{uptowi}, it is clear that the moduli space $\MPP(-F,1)$ is the quotient of $S$ by a $\C^*$ action of weight 1 on $(u,v) \in \C^2$. Hence we have

\begin{theorem}\label{explicitdescriptionmodulispace1}
$\MPP(-F,1)$ is a 7-dimensional algebraic variety whose singular locus are the points $(E,0)$ for any non-trivial extension $E$.
\end{theorem}

\begin{remark}
For the cases where $c_1=0, -C_0, -C_0-F$ and $c_2=1$, one can use analogous arguments to the ones presented here in order to describe the underlying bundles of the semistable co-Higgs pairs with those Chern classes, and this was done in \cite{AVC}. Nonetheless, the description of the Higgs fields, and thus a full picture of the moduli space in those cases, would also require a better understanding of the Higgs fields they admit. We expect that similar arguments to the ones presented here would achieve this goal.
\end{remark}

\section{Spectral correspondence and Hitchin map}\label{ssP1xP1}

In this section, we present spectral surfaces over $\P \times \P$ and discuss the Hitchin correspondence in this setting. Although the general theory for spectral surfaces has already been developed by Simpson in \cite{Simpson2}, here we exhibit the construction explicitly, providing the equations that cut out the spectral surface in the generic case. 

The Hitchin map $H$ that goes from the moduli space of semistable rank 2 co-Higgs bundles over $\P \times \P$, $\MMM$, to the global sections of $\text{S}^2(T)=\OO(4,0) \p \OO(2,2) \p \OO(0,4)$ is defined as follows:
\begin{displaymath}
\begin{array}{cccc}
H: & \MMM & \to & \HH^0(\text{S}^2(T)) \\
 & (E, \Phi) & \mapsto & \ch \Phi.
\end{array}
\end{displaymath}

Here we are identifying $\ch \Phi$ with $\det \Phi \in \HH^0(\text{S}^2(T))$, as $\ch \Phi= \eta^2(y) + \det \Phi$ (since $\Phi$ is trace-free), where $\eta$ denotes the tautological section of the pullback of $T$ to its own total space.
Explicitly, the Hitchin map is given as follows: 
Let $(E, \Phi) \in \MMM$, then working on an open set $\U$, we can write
\begin{displaymath}
\Phi=\Phi_1+\Phi_2=\left(\begin{array}{cc}
A_1 & B_1 \\
C_1 & -A_1 \\
\end{array}\right)
+
\left(\begin{array}{cc}
A_2 & B_2 \\
C_2 & -A_2 \\
\end{array}\right),
\end{displaymath}
where $A_1, B_1, C_1 \in \HH^0(\U,\OO(2,0))$ and $A_2, B_2, C_2 \in \HH^0(\U,\OO(0,2))$. Then
\begin{displaymath}
H(E,\Phi)=(\det \Phi_1, -2A_1A_2-2B_1C_2, \det \Phi_2) \in \HH^0(\OO(4,0) \p \OO(2,2) \p \OO(0,4)).
\end{displaymath}
We first note that the Hitchin map is not surjective. Indeed, if $H(E,\Phi)=(\rho_1, \rho_{1,2}, \rho_2)$, then by the above equation and the integrability of $\Phi$, we see that $\rho_{1,2}^2=4\rho_1\rho_2$, and so $H$ is clearly not onto. 

\begin{definition}
Let $(E,\Phi) \in \MMM$. The spectral surface $S_{\rho}$ associated to $\rho=\ch \Phi$, is given by those points $y \in \Tot(T)$ such that 
\begin{displaymath}
\ch \Phi(y) = \eta^2(y) + \det \Phi( \theta(y) )=0,
\end{displaymath}
where $\theta: \Tot(T) \to \P \times \P$. We equip $S_{\rho}$ with the restriction $\theta |_{S_{\rho}}$. 
\end{definition}
The equations of $S_{\rho}$ can be written as follows: if $y=(y_1,y_2) \in \Tot(T)$ and $\rho=(\rho_1, \rho_{1,2}, \rho_2)$, then the spectral surface is given by
\begin{equation}\label{spectralsurfaceeqn}
 \left\{
\begin{array}{c}
\eta_1^2(y_1) + \rho_1=0 \\
\eta_2^2(y_2) + \rho_2 =0 \\
2\eta_1(y_1)\eta_2(y_2) + \rho_{1,2}=0,
\end{array}
\right.
\end{equation}
where $\eta_1$ and $\eta_2$ are the tautological sections of the pullback of $\OO(2,0)$ and $\OO(0,2)$, respectively, to their own total spaces.

Note that multiplying the first equation by $\eta_2^2(y_2)$, and using the second equation and the fact that $\rho_{1,2}^2=4\rho_1\rho_2$, yields 
\begin{align*}
0 & = (\eta_1^2(y_1)+ \rho_1)\eta_2^2(y_2) \\
   & = \eta_1^2(y_1)\eta_2^2(y_2)-\rho_1\rho_2  \\
   & = (\eta_1(y_1)\eta_2(y_2) + \rho_{1,2}/2)(\eta_1(y_1)\eta_2(y_2) - \rho_{1,2}/2).
\end{align*}
Thus the first two equations yield a reducible surface in $\Tot(T)$. Clearly, as the third equation appears as one of the factors above, it cuts out a 2-dimensional subvariety of this surface. 
 
The remark below is an observation on what the integrability of $\Phi$ entails geometrically (in terms of eigenspaces), which follows immediately from this basic lemma from linear algebra (we include a proof here, as we were not able to find a proper reference):

\begin{lemma}
Suppose $M_1$ is an $n \times n$ complex matrix with non-repeated eigenvalues. If $M_2$ is such that $[M_1,M_2]=0$, then $M_1$ and $M_2$ have the same eigenvectors. In particular, if $M_2$ has non-repeated eigenvalues, then $M_1$ and $M_2$ have the same eigenspaces.
\end{lemma} 

\begin{proof}
Let $\lambda$ be an eigenvalue of $M_1$ and $v \neq 0$ be in the eigenspace corresponding to $\lambda$. Then, we have that
\begin{align*}
\pmb{0} & =  \pmb{0}v \\
             & =  [M_1,M_2]v \\
             & = M_1M_2v-M_2M_1v \\
             & = M_1(M_2v)-\lambda(M_2v),
\end{align*}
and so $M_1(M_2v) = \lambda(M_2v)$. This implies that $M_2v$ is an element of the eigenspace corresponding to $\lambda$, and so it can be written as a complex multiple of $v$. Hence $M_2v=\lambda'v$, and the result follows.
\end{proof}

\begin{remark}\label{integrabilitygeometric}
We have seen that the integrability of $\Phi$ is equivalent to $[\Phi_1,\Phi_2]=0$. Thus, from the above lemma, for those points of $\P \times \P$ where $\Phi_1$ and  $\Phi_2$ have non-repeated eigenvalues, we must have that $\Phi_1$ and $\Phi_2$ share the same eigenspaces.
\end{remark}
 
 Analogous to the case of curves, the elements of $S_{\rho}$ lying above a point in $\P \times \P$ are pairs where the first entry is an eigenvalue of $\Phi_1$, and the second entry is an eigenvalue of $\Phi_2$. Moreover, we claim that for generic $\rho$, $S_{\rho}$ is a double cover of $\P \times \P$. To see this, let $\lambda_i^1$ and $\lambda_i^2$ be the eigenvalues of $\Phi_i$ at an unramified point $p \in \P \times \P$. Since $\Phi$ is integrable, by Remark \ref{integrabilitygeometric}, $\Phi_1$ and $\Phi_2$ have the same eigenspaces, and so we assume that the eigenspace of $\lambda_1^j$ is equal to the eigenspace of $\lambda_2^j$ for $j=1,2$. We now check that the third equation of $S_{\rho}$ is equivalent to $(\lambda_1^i, \lambda_2^j) \in S_{\rho}$ if and only if $i=j$. In other words, the points of $S_{\rho}$, at unramified points of $\P \times \P$, are pairs of eigenvalues of $\Phi_1$ and $\Phi_2$ sharing the same eigenspace. First note that since $\Phi_1$ and $\Phi_2$ commute, $\lambda_1^j$ and $\lambda_2^j$ sharing the same eigenspaces is equivalent to  $\lambda_1^j\lambda_2^j$ being an eigenvalue of $\Phi_1\Phi_2$, and so it must satisfy the characteristic polynomial $\ch (\Phi_1\Phi_2)$:
\begin{displaymath}
\eta_1^2(\lambda_1^j)\eta_2^2(\lambda_2^j)-\text{tr}(\Phi_1\Phi_2)\eta_1(\lambda_1^j)\eta_2(\lambda_2^j)+\det(\Phi_1\Phi_2) = 0.
\end{displaymath}
After some algebraic manipulation, the above equation reduces to
\begin{displaymath}
2\eta_1(\lambda_1^j)\eta_2(\lambda_2^j) + \rho_{1,2}=0,
\end{displaymath}
which is precisely saying that $(\lambda_1^j,\lambda_2^j)$ satisfies the third equation. We can thus conclude that the points in $S_{\rho}$ lying above $p \in \P \times \P$ are $(\lambda_1^1,\lambda_2^1)$ and $(\lambda_1^2,\lambda_2^2)$, showing that $S_{\rho}$ is indeed a double cover of $\P \times \P$.

\begin{remark}\label{ellipticfibrations} Let $\rho$ be generic. Unlike the curves case, in order to get a Hitchin correspondence, one needs to push-forward rank 1 torsion free sheaves over $S_{\rho}$ instead of only elements of $\Pic(S_{\rho})$ (see \cite{Simpson1,Simpson2}). 
\end{remark}

Now, we aim to show that, the underlying bundle of the generic elements of $\MMM$ are indecomposable.

\begin{lemma}\label{decomphi1orphi2}
Let $E=L_1 \p L_2$ be a decomposable rank 2 vector bundle over $\P \times \P$. 
\begin{enumerate} 
\item Suppse $\mu(L_1) > \mu(L_2)$. If $(E, \Phi=\Phi_1 + \Phi_2)$ is a semistable co-Higgs pair, then $\Phi_1=0$ or $\Phi_2=0$.
\item Suppose $\mu(L_1)=\mu(L_2)$. Then, either $\det \Phi_1$ is non-generic in $\HH^0(\OO(4,0))$ or $\det \Phi_2$ is non-generic in $\HH^0(\OO(0,4))$.
\end{enumerate}
\end{lemma}

\begin{proof}
Let $L_1=\OO(a_1,b_1)$ and $L_2=\OO(a_2,b_2)$. Then, any element of $\HH^0(\End_0 E \tt T)$ is of the form 
\begin{displaymath}
\Phi=\Phi_1+\Phi_2=\left(\begin{array}{cc}
A_1 & B_1 \\
C_1 & -A_1 \\
\end{array}\right)
+
\left(\begin{array}{cc}
A_2 & B_2 \\
C_2 & -A_2 \\
\end{array}\right),
\end{displaymath}
with $A_1 \in \HH^0(\OO(2,0))$, $B_1 \in \HH^0(\OO(a_1-a_2+2,b_1-b_2)$, $C_1 \in \HH^0(\OO(a_2-a_1+2,b_2-b_1))$ and $A_2 \in \HH^0(\OO(0,2))$, $B_2 \in \HH^0(\OO(a_1-a_2,b_1-b_2+2))$, $C_2 \in \HH^0(\OO(a_2-a_1,b_2-b_1+2))$.

\noindent 1. Suppose $\mu(L_1) > \mu(L_2)$. We have two cases to consider. If $a_1 > a_2$, then any element in $\HH^0(\End_0 E \tt T)$ is such that $C_2=0$. If we were to have a Higgs field $\Phi$ for $E$ such that $(E, \Phi)$ is semistable, then $C_1$ must not be identically zero, for otherwise it would leave $L_1$ invariant, contradicting semistability. The integrability condition, equations (\ref{integrability}), implies that $A_2=B_2=0$. Hence, $\Phi=\Phi_1$ with non-zero $C_1$. Similarly, if $b_1 > b_2$, we get that $\Phi=\Phi_2$ with non-zero $C_2$.

\noindent 2. Suppose $\mu(L_1)=\mu(L_2)$. It is enough to consider the following three cases:
\begin{enumerate}
\item[(i)] If $a_1 > a_2$ and $b_2 > b_1$, then any element in $\HH^0(\End_0 E \tt T)$ is such that $B_1=0$ and $C_2=0$. Then $\det \Phi_1=-A_1^2$ and $\det \Phi_2=-A_2^2$, so they are non-generic in $\HH^0(\OO(4,0))$ and $\HH^0(\OO(0,4))$, respectively. 
\item[(ii)] If $a_2 > a_1$ and $b_1 > b_2$, then any element in $\HH^0(\End_0 E \tt T)$ is such that $C_1=0$ and $B_2=0$. Then $\det \Phi_1=-A_1^2$ and $\det \Phi_2=-A_2^2$, so they are non-generic in $\HH^0(\OO(4,0))$ and $\HH^0(\OO(0,4))$, respectively. 
\item[(iii)] If $a_1=a_2$ and $b_1=b_2$, then any element in $\HH^0(\End_0 E \tt T)$ is such that $A_1, B_1, C_1$ are elements of $\HH^0(\OO(2,0))$ and $A_2, B_2, C_2$ are elements of $\HH^0(\OO(0,2))$. Note that we may assume that at least one entry in either $\Phi_1$ or $\Phi_2$ is non-zero, for otherwise the result follows. Without loss of generality let us assume that $A_2 \neq 0$. By the integrability condition, equations (\ref{integrability}) imply that if $A_2 \neq 0$, we can pick a point $p_2 \in \P$ , which is not a zero of $A_2$, evaluating both $A_1B_2=B_1A_2$ and $C_1A_2=A_1C_2$ on $p=(z_1,p_2)$, we get $B_1=uA_1$ and $C_1=vA_1$ for $u,v \in \C$. Hence, $\det \Phi_1=-(1+uv)A_1^2$, which is non-generic in $\HH^0(\OO(4,0))$. 
\end{enumerate}
\end{proof}

\begin{proposition}\label{fibresoftheHitchinmap}
If $\rho$ is generic, then $H^{-1}(\rho)$ does not contain co-Higgs pairs where the underlying bundle is decomposable. In particular, for $(E, \Phi) \in \MMM$ generic, $E$ is not decomposable. 
\end{proposition}

\begin{proof}
Let $\rho$ be generic and assume that $H^{-1}(\rho)$ contains a pair with decomposable underlying bundle. Then, by Lemma \ref{decomphi1orphi2} (1.), either $\rho=(\rho_1,0,0)$ or $\rho=(0,0,\rho_2)$, or, by Lemma \ref{decomphi1orphi2} (2.), $\rho=(\rho_1,\rho_{1,2},\rho_2)$ with either $\rho_1$ or $\rho_2$ non-generic. Hence, $\rho$ is not generic.
\end{proof}

Let us now discuss spectral surfaces in the case where either $\Phi_1$ or $\Phi_2$ is zero. We will be interested in the cases when $\rho=(\rho_1,0,0)$ or $(0,0,\rho_2) \in \HH^0(\text{S}^2(T))$, where $\rho_1$ and $\rho_2$ are generic in $\HH^0(\OO(4,0))$ and $\HH^0(\OO(0,4))$, respectively. 

When $\rho=(\rho_1,0,0)$, and $\rho_1$ is generic, any Higgs field $\Phi$ of a co-Higgs pair in the fibre of the Hitchin map above $\rho$ must have the form $\Phi=\Phi_1$. To see this, let $\Phi$ be a Higgs field such that $\det \Phi = \rho$. Since $\det \Phi_2=0$, we have that $\lambda=0$ is an eigenvalue of $\Phi_2$ of algebraic multiplicity $2$. Also, above all points where $\det \Phi_1 \neq 0$, we have a basis of eigenvectors for $\Phi_1$. By the integrability of $\Phi$ and Remark \ref{integrabilitygeometric}, this is also a basis of eigenvectors of $\Phi_2$. Hence, $\Phi_2$ is diagonalizable and thus the zero matrix at all such points. Hence, $\Phi_2=0$. Moreover, in this case, the equations of the spectral surface reduce to
\begin{equation}
\left\{
\begin{array}{c}
\eta_1^2(y_1) + \det \Phi_1 = 0 \\
\eta_2^2(y_2) = 0.
\end{array}
\right.
\end{equation}
Hence, 
\begin{displaymath}
S_{\rho} = X_{\rho_1} \times \P,
\end{displaymath} 
where $X_{\rho_1}$ is the spectral curve associated to $\rho_1$ (we view $\rho_1$ as an element of $\HH^0(\P, \OOO(4))$), which is an elliptic curve (see \cite[Section~5]{SR2}). Also, the projection $\theta: S_{\rho} \to \P \times \P$ is given by $(\pi, \Id_{\P})$, where $\pi: X_{\rho} \to \P$.

Similar observations can be made when $\rho=(0,0,\rho_2)$ is generic (in particular $S_{\rho}= \P \times X_{\rho_2}$). Consequently, in both of these cases, we have a Hitchin correspondence on the spectral surface coming from the correspondence on the spectral curve. More precisely,

\begin{proposition}\label{correspondencepushingdownbundles}
Suppose $\rho=(\rho_1,0,0) \in \HH^0(\SSS^2(T))$ with $\rho_1$ generic. Then, there is a Hitchin correspondence between the line bundles of $S_{\rho}$ and the elements $(E, \Phi)$ of $\MMM$ with underlying bundle of the form $E= \OO(a,m) \p \OO(b,m)$ and $\Phi= \Phi_1 \in \HH^0(\End_0 E \tt \OO(2,0))$.
\end{proposition}

\begin{proof}
Let $M$ be a line bundle over $S_{\rho}$, then $M$ is of the form $\Pr_1^*L \tt \Pr_2^* \OOO(m)$ (see \cite[Chapter~3, Section~12]{H}), where $L$ is a line bundle over $X_{\rho}$ and $m \in \ZZ$, and $\Pr_1, \Pr_2$ are the projections of $S_{\rho}$ to $X_{\rho}$ and $\P$, respectively. From the commutative diagram
$$
\xymatrix{
S_{\rho}=X_{\rho} \times \P \ar[d]^{\Pr_1}\ar[rr]^{(\pi,\Id_{\P})}&&\P \times \P \ar[d]_{\pr_1} \\
X_{\rho} \ar[rr]^{\pi}&& \P
}
$$
we see that $(\pi,\Id_{\P})_*(\Pr_1^*(L))=\pr_1^*(\pi_*(L))=\OO(a,0) \p \OO(b,0)$ for some $a , b \in \ZZ$ such that $\pi_*(L)= \OOO(a) \p \OOO(b)$. Similarly, from the commutative diagram
$$
\xymatrix{
S_{\rho}=X_{\rho} \times \P \ar[d]^{\Pr_2}\ar[rr]^{(\pi,\Id_{\P})}&&\P \times \P \ar[d]_{\pr_2} \\
\P \ar[rr]^{\Id_{\P}}&& \P
}
$$
we see that $(\pi,\Id_{\P})_*(\Pr_2^*(\OOO(m)))=\pr_2^*(\Id_{\P}*(\OOO(m)))=\OO(0,m)$. Therefore, $\theta_* M= \OO(a,m) \p \OO(b,m)$. Moreover, since the multiplication of elements in $M$ by elements in $S_{\rho}$ maps to $M \tt \OO(2,0)$, the push-forward of $- \tt \eta$ yields a Higgs field $\Phi$ with $\Phi=\Phi_1$. The Higgs field $\Phi_1$ is the pullback of the Higgs field obtained by pushing-down the multiplication map of $L$.

On the other hand, if we start with something of the form $\OO(a,m) \p \OO(b,m)$, to find the corresponding line bundle over $S_{\rho}$ we first find the line bundle $L$ over $X_{\rho}$ corresponding to $(\OOO(a) \p \OOO(b), \Phi_1)$. Then we tensor the pullback of the latter with $\Pr_2^*\OOO(m)$.
\end{proof}

\begin{remark}\label{remarkcorrespondencepushingdownbundles}
A similar result holds when $\rho$ is of the form $(0,0, \rho_2)$ with $\rho_2$ generic.
\end{remark}

\bibliographystyle{plain}

\renewcommand*{\bibname}{References}
\addcontentsline{toc}{section}{\textbf{References}}
\bibliography{UWBIBB}


\end{document}